\documentclass[11pt]{amsart}

\usepackage{amssymb}
\usepackage[numbers,sort&compress]{natbib}
\usepackage{xcolor}
\usepackage{enumitem}
\usepackage{tikz}
\usepackage{pgfplots}
\usepgfplotslibrary{groupplots}
\pgfplotsset{compat=1.18}
\usepackage{hyperref}

\hypersetup{
  pdfauthor={Qian Ai, Xiang Fang, Shengzhao Hou},
  pdftitle={Hardy--Szeg\H{o} Point Processes: Large Deviations and Strong Szeg\H{o} Asymptotics},
  pdfsubject={Large deviations and strong Szeg\H{o} asymptotics for Hardy--Szeg\H{o} point processes},
  pdfkeywords={Hardy--Szeg\H{o} zeros, determinantal point processes, large deviations, strong Szeg\H{o} asymptotics, Wiener--Hopf operators, Fredholm determinants}
}

\definecolor{myblue}{RGB}{40,90,180}
\definecolor{myred}{RGB}{190,50,50}
\definecolor{myorange}{RGB}{220,130,30}
\definecolor{mypurple}{RGB}{120,70,170}
\definecolor{mygreen}{RGB}{40,140,80}

\pgfmathdeclarefunction{atanhr}{1}{%
  \pgfmathparse{0.5*ln((1+#1)/(1-#1))}%
}

\pgfmathdeclarefunction{atanrad}{1}{%
  \pgfmathparse{atan(#1)*pi/180}%
}

\title[Hardy--Szeg\H{o} Point Processes]
{Hardy--Szeg\H{o} Point Processes:
Large Deviations and Strong Szeg\H{o} Asymptotics}

\author{Qian Ai}
\address{School of Mathematical Sciences, Soochow University,
Suzhou 215006, P. R. China}
\email{20234007001@stu.suda.edu.cn}

\author{Xiang Fang}
\thanks{X. F. was supported by National Science and Technology Council
(No.~114-2115-M-A49-003-MY3).}
\address{National Yang Ming Chiao Tung University,
Hsinchu, Taiwan (R.O.C.)}
\email{xfang@nycu.edu.tw}

\author{Shengzhao Hou}
\thanks{S. H. was supported by National Natural Science Foundation of China
(No.~12371133).}
\address{School of Mathematical Sciences, Soochow University,
Suzhou 215006, P. R. China}
\email{shou@suda.edu.cn}



\numberwithin{equation}{section}

\theoremstyle{plain}
\newtheorem{theorem}{Theorem}[section]
\newtheorem{lemma}[theorem]{Lemma}
\newtheorem{proposition}[theorem]{Proposition}
\newtheorem{corollary}[theorem]{Corollary}

\theoremstyle{definition}
\newtheorem{definition}[theorem]{Definition}
\newtheorem{example}[theorem]{Example}

\theoremstyle{remark}
\newtheorem{remark}[theorem]{Remark}

\begin{document}

\begin{abstract}
We study exponential-scale fluctuations of the Hardy--Szeg\H{o} zero process, investigated by Peres and Vir\'ag in the disk setting, in its upper-half-plane realization, which reveals a different probabilistic geometry. This conformally invariant determinantal zero process is equivalent to its disk realization, but the upper-half-plane coordinates make real-translation invariance explicit and single out long horizontal windows as natural observables. For the vertical window from height one to height \(a>1\), let \(N_a(L)\) denote the number of zeros in the corresponding horizontal window of length \(L\). We identify the limiting scaled log-moment generating function explicitly. As a consequence, we prove a large deviations principle for \(N_a(L)/L\), with rate function given by the Legendre transform of this limit. We also prove a strong Szeg\H{o} expansion for the log-moment generating function, including an explicit order-one correction, locally uniformly in the natural complex strip. The proof uses the planar determinantal structure before projection, reduces the problem to a one-dimensional Fredholm determinant, and combines fixed-power trace asymptotics with a Wiener--Hopf comparison.
\end{abstract}

\maketitle

\section{Introduction and main results}
\label{sec:introduction-main-results}

\noindent\textbf{The Hardy--Szeg\H{o} model and the upper-half-plane geometry.} The Hardy--Szeg\H{o} zero process is a canonical conformally invariant determinantal point process arising from zeros of Gaussian analytic functions. Its disk realization was studied by Peres and Vir\'ag~\cite{PeresVirag2005} as the zero set of the independent identically distributed Gaussian power series on the unit disk. In this paper we study this conformally invariant process in its upper-half-plane realization, which makes horizontal stationarity explicit: the process is invariant under real translations. This makes horizontal strips and long boundary-parallel windows natural observables, and leads to asymptotic questions that are not apparent in disk coordinates.

Let
\(
        \mathbb H:=\{z\in\mathbb C:\operatorname{Im}z>0\},
\)
and let \(\mathrm dA\) denote planar Lebesgue measure.  We consider the
determinantal point process on \((\mathbb H,\mathrm dA)\) with Hermitian
correlation kernel
\[
        K_{\mathbb H}(z,w)
        =
        -\frac{1}{\pi(z-\overline w)^2},
        \qquad z,w\in\mathbb H.
\]
Equivalently, this process is the zero set of the Hardy--Szeg\H{o} Gaussian
analytic function
\[
        f_{\mathbb H}(z)
        =
        (\varphi'(z))^{1/2}
        \sum_{n=0}^{\infty}\xi_n\varphi(z)^n,
        \qquad z\in\mathbb H,
\]
where \((\xi_n)_{n\geq0}\) are independent standard complex Gaussian random
variables, \(\varphi:\mathbb H\to \mathbb D:=\{z\in\mathbb C: |z|<1\}\) is the Cayley map
\(
        \varphi(z):=\frac{z-i}{z+i},
\)
and a holomorphic branch of \((\varphi')^{1/2}\) is fixed.  We write
\(\mathcal Z:=Z(f_{\mathbb H})\) for the zero set.  The explicit kernel and
the horizontal stationarity are the two structural features used throughout
the paper.  A more detailed account of the analytic function and its zero
process is given in Section~\ref{sec:model-reduction}.

\medskip

\noindent\textbf{Long horizontal windows and projection.}
The upper-half-plane realization suggests a natural macroscopic question: how
does the number of zeros fluctuate in a boundary-parallel window whose
horizontal length tends to infinity while its height range remains fixed? For a bounded vertical interval
\(I\subset(0,\infty)\) and \(L>0\), define
\[
        N_I(L)
        :=
        \#\{z\in\mathcal Z:
        0\leq \operatorname{Re}z\leq L,\ 
        \operatorname{Im}z\in I\}.
\]
Thus \(N_I(L)\) counts Hardy--Szeg\H{o} zeros in the horizontal window
\([0,L]\times I\).  The regime considered in this paper is \(L\to\infty\),
with the vertical interval \(I\) fixed.

Equivalently, these counts may be encoded by projecting the zeros in the strip
\(\mathbb R\times I\) onto the real line.  We set
\[
        \mathcal X_I
        :=
        \sum_{\substack{z\in\mathcal Z\\ \operatorname{Im}z\in I}}
        \delta_{\operatorname{Re}z},
        \qquad
        N_I(L)=\mathcal X_I([0,L]).
\]
The process \(\mathcal X_I\) is therefore a one-dimensional stationary point
process obtained as a natural horizontal observation of the planar
Hardy--Szeg\H{o} zero process.  This projection viewpoint is useful but also
delicate.  Determinantal point processes are powerful because their correlation
functions and generating functionals have explicit determinant forms, but such
structure is not automatically preserved under geometric operations.  A
projection collapses one coordinate and can change the local structure of the
process, while retaining nontrivial dependence inherited from the underlying
two-dimensional determinantal process.  For this reason, our proof does not
treat the projected process as a stand-alone determinantal point process.
Instead, we use the planar determinantal structure first and integrate over the
height variable only after the determinant has been reduced to a suitable
one-dimensional operator.

\medskip

\noindent\textbf{From Gaussian fluctuations to exponential asymptotics.}
Horizontal window counts for the projected Hardy--Szeg\H{o} zero process were
studied at the Gaussian scale in~\cite{AiGuoZhou2026}, where macroscopic
Gaussian fluctuations and a functional central limit theorem were established.
The present paper continues this line of investigation beyond the Gaussian
regime.  We study exponential moments, large deviations, and strong
Szeg\H{o} asymptotics for the long-window counts.

More precisely, for a bounded vertical interval \(I\subset(0,\infty)\), we
analyze the scaled logarithmic moment generating functions
\[
        \Lambda_{I,L}(\theta)
        :=
        \frac{1}{L}
        \log\mathbb E\exp(\theta N_I(L)).
\]
By the dilation invariance of the Hardy--Szeg\H{o} zero process, it is enough
to treat the reference windows \(I_a=[1,a]\), \(a>1\). For these windows we write
\[
        \mathcal X_a:=\mathcal X_{I_a},
        \qquad
        N_a(L):=\mathcal X_a([0,L]),
        \qquad
        \Lambda_{a,L}(\theta)
        :=
        \frac{1}{L}\log\mathbb E\exp(\theta N_a(L)).
\]
This is the notation used in the statements below.

\subsection{Main results}

We now state the main results for the normalized reference family \(I_a=[1,a]\), \(a>1\).

The first result identifies the limiting scaled log-moment generating function.
We write \(S_{\pi}:=\{\theta\in\mathbb C:|\operatorname{Im}\theta|<\pi\}\) for
the natural strip of holomorphy.

\begin{theorem}[Szeg\H{o} limit]
\label{thm:szego-limit-reference}
Let \(a>1\) and \(I_a=[1,a]\). For every \(\theta\in\mathbb{R}\), the limit
\[
\Lambda_a(\theta):=\lim_{L\to\infty}\Lambda_{a,L}(\theta)
\]
exists, and the convergence is locally uniform in \(\theta\in\mathbb{R}\). Moreover,
\[
\Lambda_a(\theta)
=
\frac{1}{2\pi}
\int_0^{\infty}
\log\left(
1+(e^{\theta}-1)(e^{-2\xi}-e^{-2a\xi})
\right)
\,\mathrm{d}\xi.
\]
The function \(\Lambda_a\) is real analytic on \(\mathbb{R}\) and extends holomorphically to \(S_{\pi}\), where the logarithm is the branch obtained by analytic continuation from \(\theta=0\). Furthermore,
\begin{equation}
\label{eq:lambda-a-derivatives-at-zero}
\Lambda_a'(0)
=
\frac{1}{4\pi}\left(1-\frac{1}{a}\right),
\qquad 
\Lambda_a''(0)
=
\frac{(a-1)(a+3)}{8\pi a(a+1)}.
\end{equation}
\end{theorem}

The derivative \(\Lambda_a'(0)\) is the intensity of \(\mathcal{X}_a\), and \(\Lambda_a''(0)\) is the linear variance density of \(N_a(L)\). 

The convexity and steepness properties of \(\Lambda_a\) then allow us to apply
the G\"artner--Ellis theorem \cite{Ellis1984,Gaertner1977}, yielding the large
deviations principle.

\begin{theorem}[Large deviations principle]
\label{thm:ldp-reference}
Let \(a>1\) and \(I_a=[1,a]\). Then \(N_a(L)/L\) satisfies a large deviations principle on \([0,\infty)\) with speed \(L\) and good rate function
\[
J_a(x):=\sup_{\theta\in\mathbb{R}}\{\theta x-\Lambda_a(\theta)\},
\qquad x\geq0.
\]
That is, for every Borel set \(A\subset[0,\infty)\),
\[
\begin{aligned}
-\inf_{x\in A^\circ}J_a(x)
&\le \liminf_{L\to\infty}\frac{1}{L}\log\mathbb{P}\left(\frac{N_a(L)}{L}\in A\right) \\
&\le \limsup_{L\to\infty}\frac{1}{L}\log\mathbb{P}\left(\frac{N_a(L)}{L}\in A\right) \le -\inf_{x\in \overline{A}}J_a(x),
\end{aligned}
\]
where \(A^\circ\) and \(\overline A\) denote the interior and closure of \(A\) in the relative topology of \([0,\infty)\). Moreover, \(J_a\) is strictly convex on \((0,\infty)\) and vanishes only at \(\lambda_a=\frac{1}{4\pi}\left(1-\frac{1}{a}\right)\). 
\end{theorem}

\begin{remark}
The LDP is naturally stated on \([0,\infty)\), since \(N_a(L)/L\geq0\)
almost surely.  At the left endpoint, the rate function is understood by
continuity:
\(
        J_a(0):=\lim_{x\downarrow0}J_a(x)
        =-\lim_{\theta\to-\infty}\Lambda_a(\theta),
\)
which is consistent with the Legendre--Fenchel representation
\(J_a(x)=\sup_{\theta\in\mathbb R}\{\theta x-\Lambda_a(\theta)\}\).
The same statement may also be formulated on \(\mathbb R\) by extending the rate function as \(\widetilde J_a(x)=J_a(x)\) for \(x\geq0\) and \(\widetilde J_a(x)=+\infty\) for \(x<0\).
\end{remark}

The next result refines the leading logarithmic asymptotics to a strong Szeg\H{o} expansion.

\begin{theorem}[Strong Szeg\H{o} expansion]
\label{thm:strong-szego-reference}
Let \(a>1\) and \(I_a=[1,a]\). There exists a holomorphic function \(\mathcal C_a:S_\pi\to\mathbb C\) such that, for every compact \(K\subset S_\pi\),
\[
\log\mathbb E\exp(\theta N_a(L))
=
L\Lambda_a(\theta)+\mathcal C_a(\theta)+o(1),
\qquad L\to\infty,
\]
uniformly for \(\theta\in K\).
\end{theorem}

The correction term \(\mathcal C_a\) is constructed through the structural
decomposition
\[
        \mathcal C_a(\theta)
        =
        \mathcal C_a^{\operatorname{WH}}(\theta)+\Phi_a(\theta),
\]
where \(\mathcal C_a^{\operatorname{WH}}\) is the canonical Wiener--Hopf
correction and \(\Phi_a\) is the holomorphic discrepancy limit between the
original reduced operator and the canonical comparison model. This decomposition
also identifies the quadratic term: Proposition~\ref{prop:correction-second-order}
shows that
\[
        \mathcal C_a(\theta)
        =
        \frac{1}{2\pi^2}
        \log\frac{(a+1)^2}{4a}\,\theta^2
        +
        O(|\theta|^3),
        \qquad \theta\to0.
\]
The same \(O(1)\) coefficient is reflected in the variance asymptotics obtained
by a direct second-order determinantal calculation. Namely,
Corollary~\ref{cor:variance-asymptotics} gives
\[
        \operatorname{Var}(N_a(L))
        =
        L\frac{(a-1)(a+3)}{8\pi a(a+1)}
        +
        \frac{1}{\pi^2}\log\frac{(a+1)^2}{4a}
        +
        o(1).
\]

We now pass from the reference family to arbitrary bounded intervals. If
\(I=[\alpha,\beta]\subset(0,\infty)\) and \(a=\beta/\alpha>1\), then
\(I=\alpha I_a\). The Hardy--Szeg\H{o} kernel is homogeneous under positive
dilations,
\[
K_{\mathbb H}(\alpha z,\alpha w)=\alpha^{-2}K_{\mathbb H}(z,w),
\qquad \alpha>0,
\]
and the zero process \(Z\) is therefore dilation-invariant in distribution; see
Section~\ref{sec:model-reduction}. Lemma~\ref{lem:count-scaling} records the
finite-\(L\) scaling relations
\[
N_I(L)\stackrel{\mathrm d}=N_a(L/\alpha),
\qquad
\Lambda_{I,L}(\theta)=\alpha^{-1}\Lambda_{a,L/\alpha}(\theta).
\]
For the limiting objects we write
\[
\Lambda_I(\theta):=\lim_{L\to\infty}\Lambda_{I,L}(\theta),
\qquad
J_I(x):=\sup_{\theta\in\mathbb R}\{\theta x-\Lambda_I(\theta)\},
\]
whenever the limit exists, and \(\mathcal C_I(\theta)\) for the strong
Szeg\H{o} correction associated with \(I\).

\begin{corollary}
\label{cor:general-interval-scaling}
Let \(I=[\alpha,\beta]\subset(0,\infty)\), \(0<\alpha<\beta<\infty\), and set \(a=\beta/\alpha>1\). Then the corresponding limiting objects satisfy
\[
\Lambda_I(\theta)=\alpha^{-1}\Lambda_a(\theta),
\qquad
J_I(x)=\alpha^{-1}J_a(\alpha x),
\qquad
\mathcal{C}_I(\theta)=\mathcal{C}_a(\theta).
\]
Moreover, the variables \(N_I(L)/L\) satisfy an LDP on \([0,\infty)\) with speed \(L\) and good rate function \(J_I\), and
\[
\log\mathbb{E}\exp(\theta N_I(L))
=
L\Lambda_I(\theta)+\mathcal{C}_I(\theta)+o(1)
\]
locally uniformly for \(\theta\in S_{\pi}\).
\end{corollary}

\subsection{\texorpdfstring{A closed-form example: \(a=2\)}{A closed-form example: a=2}}

For \(a=2\), corresponding to \(I=[1,2]\), the limiting log-mgf has the closed form
\[
\Lambda_2(\theta)
=
\frac{1}{2\pi}\operatorname{arctanh}^2
\sqrt{\frac{e^\theta-1}{e^\theta+3}},
\]
with either local square-root branch; the value is branch-independent. The associated rate function is parametrized by \(z\in i[0,1/\sqrt3]\cup[0,1)\), with continuous extension at \(z=0\), through
\[
x(z)
=
\frac{1+3z^2}{8\pi z}\operatorname{arctanh}z\ge 0,
\]
and 
\[J_2(x(z))
=
x(z)\log\left(\frac{1+3z^2}{1-z^2}\right)
-
\frac{1}{2\pi}\operatorname{arctanh}^2z.
\]
In particular, \(J_2\left(\frac{1}{8\pi}\right)=0\) and \(J_2(0)=\frac{\pi}{72}\). For general integer values of \(a\), a related dilogarithmic representation of
\(\Lambda_a\) is recorded in Example~\ref{ex:integer-shape-explicit}.

\begin{figure}[htbp]
\centering
\begin{tikzpicture}
\begin{groupplot}[
    group style={group size=2 by 1, horizontal sep=1.7cm},
    width=0.47\textwidth,
    height=0.36\textwidth,
    axis lines=middle,
    grid=both,
    minor tick num=1,
    xlabel style={below right},
    ylabel style={above left},
    tick label style={font=\small},
    label style={font=\small},
    every axis/.append style={
        line width=0.4pt,
    },
]

\nextgroupplot[
    xlabel={$\theta$},
    ylabel={$\Lambda_2(\theta)$},
    xmin=-6.5,
    xmax=5.5,
    ymin=-0.18,
    ymax=1.13,
    xtick={-4,0,4},
    samples=200,
]

\addplot[thick, myblue, domain=-6.5:-0.001]
    {-(atanrad(sqrt((1-exp(x))/(exp(x)+3)))^2)/(2*pi)};

\addplot[thick, myblue, domain=0.001:5.5]
    {(atanhr(sqrt((exp(x)-1)/(exp(x)+3)))^2)/(2*pi)};

\addplot[
    only marks,
    mark=*,
    mark size=2pt,
    myred,
    mark options={fill=myred,draw=myred}
] coordinates {(0,0)};

\node[anchor=north west, font=\small, myred] at (axis cs:0,0)
    {$0$};

\addplot[
    dashed,
    thick,
    mygreen,
    domain=-6:5.5,
    samples=2
] {-pi/72};

\node[anchor=south east, font=\small, mygreen, yshift=2pt]
    at (axis cs:5.5,{-pi/72}) {$-\frac{\pi}{72}$};

\nextgroupplot[
    xlabel={$x$},
    ylabel={},
    xmin=-0.03,
    xmax=0.32,
    ymin=-0.09,
    ymax=0.56,
    xtick={0,0.1,0.2,0.3},
    samples=200,
    clip=false,
]

\node[font=\small, anchor=south]
    at (axis description cs:0.5,1.03)
    {$J_2(x)$};

\addplot[thick, myorange, domain=0.001:0.5772]
(
    {((1-3*x^2)*atanrad(x))/(8*pi*x)},
    {(((1-3*x^2)*atanrad(x))/(8*pi*x))
      *ln((1-3*x^2)/(1+x^2))
      + (atanrad(x)^2)/(2*pi)}
);

\addplot[thick, myorange, domain=0.001:0.96]
(
    {((1+3*x^2)*atanhr(x))/(8*pi*x)},
    {(((1+3*x^2)*atanhr(x))/(8*pi*x))
      *ln((1+3*x^2)/(1-x^2))
      - (atanhr(x)^2)/(2*pi)}
);

\addplot[
    only marks,
    mark=*,
    mark size=2pt,
    myred,
    mark options={fill=myred,draw=myred}
] coordinates {(0,{pi/72})};

\node[anchor=south west, font=\small, myred]
    at (axis cs:0,{pi/72})
    {$\frac{\pi}{72}$};

\addplot[
    only marks,
    mark=*,
    mark size=2pt,
    myblue,
    mark options={fill=myblue,draw=myblue}
] coordinates {({1/(8*pi)},0)};

\node[anchor=north west, font=\small, myblue]
    at (axis cs:{1/(8*pi)},0)
    {$\frac{1}{8\pi}$};

\end{groupplot}
\end{tikzpicture}
\caption{The limiting log-mgf $\Lambda_2$ and the associated rate function $J_2$.}
\label{fig:Lambda-J-a2}
\end{figure}
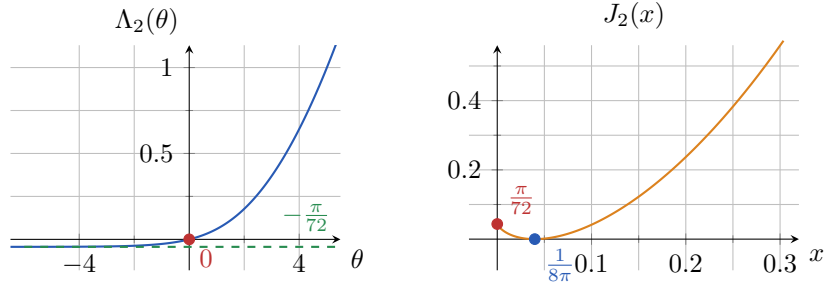

\subsection{Background, related work, and proof strategy}

The large deviations part of the paper follows the G\"artner--Ellis route.
Once the limiting scaled logarithmic moment generating function has been
identified on the whole real line, the rate function is obtained by Legendre
duality. Classical large-deviation results include Cram\'er's theorem for sums of
independent variables \cite{Cramer1938}, Sanov's theorem for empirical measures
\cite{Sanov1961}, and Schilder's theorem for Wiener paths \cite{Schilder1966}. In this respect the result is analogous in mechanism to
Cram\'er-type large deviations \cite{Cramer1938}, while the variables here come from a strongly
dependent zero process rather than from independent sums.  We use the
G\"artner--Ellis theorem in the form of
\cite{Ellis1984,Gaertner1977}; for general references on large deviations, see
\cite{DemboZeitouni2010,DenHollander2000,DeuschelStroock1989,
DupuisEllis1997,Ellis2006,FengKurtz2006,Touchette2009,Varadhan1984}. 

The random input is a zero process of a Gaussian analytic function.  Gaussian
analytic functions and their zero sets form a rich class of random point
processes; see
\cite{HoughKrishnapurPeresVirag2009,Sodin2000} for general background.  The
surrounding literature includes random polynomial zeros
\cite{BogomolnyBohigasLeboeuf1992,EdelmanKostlan1995,Friedman1990,
Hammersley1956,Kostlan1992,Kostlan1993}, zeros of random holomorphic sections
\cite{BleherShiffmanZelditch2000,Shiffman2008,ShiffmanZelditch1999}, and
translation-invariant or rigid Gaussian zero processes
\cite{Feldheim2013,Feldheim2018,GhoshPeres2017}.  Against this background, the
present paper focuses on exponential-scale and strong-asymptotic questions for
long horizontal window counts in the Hardy--Szeg\H{o} half-plane model.

The determinantal structure of the planar Hardy--Szeg\H{o} zero process is the
main algebraic input.  General references on determinantal point processes
include
\cite{Berman2014,LavancierMollerRubak2015,Macchi1975,
ShiraiTakahashi2003,Soshnikov2000}.  In the present problem, however, the
determinantal structure is used before projection.  The planar generating
functional gives a Fredholm determinant on the window
\([0,L]\times[1,a]\), and a Fourier--Laplace factorization of the
Hardy--Szeg\H{o} kernel, together with Sylvester's identity, reduces it to a
one-dimensional trace-class determinant; see
Section~\ref{sec:model-reduction} and
Lemma~\ref{lem:reduced-operator}.  The resulting operator on \(L^2(0,\infty)\)
has kernel
\[
T_{a,L}(\xi,\eta)
=
\frac{\sqrt{\xi\eta}}{\pi}
X_L(\xi-\eta)H_a(\xi+\eta),
\qquad \xi,\eta>0,
\]
as in \eqref{eq:reduced-kernel}.  Fixed-power trace asymptotics for
\(T_{a,L}\), proved in Proposition~\ref{prop:trace-moment-limit}, give the
leading Szeg\H{o} limit.  We use standard trace-ideal tools, in the sense of
\cite{Simon2005}, to control the corresponding Fredholm determinants.

The strong \(O(1)\) term is connected with classical Szeg\H{o} and
Wiener--Hopf determinant theory.  The operator \(T_{a,L}\) is close to a
Wiener--Hopf model, but it is not itself of Wiener--Hopf form because of the
factor \(H_a(\xi+\eta)\).  We therefore compare it with a canonical
Wiener--Hopf operator obtained by replacing \(H_a(\xi+\eta)\) with the
separable diagonal approximation
\(
\sqrt{H_a(2\xi)H_a(2\eta)}.
\)
The second-order contact of these two factors along the diagonal yields a
uniform trace-class discrepancy; see
Proposition~\ref{prop:trace-class-discrepancy}.  The canonical Wiener--Hopf
part of the \(O(1)\) correction is obtained in
Proposition~\ref{prop:canonical-strong-szego}, while the remaining discrepancy
is shown in Proposition~\ref{prop:discrepancy-limit} to have a holomorphic
determinant limit.  This connects the proof to classical Szeg\H{o} theory
\cite{Szego1921,Szego1952}, Toeplitz and Wiener--Hopf determinant asymptotics
\cite{BottcherSilbermann1999,Widom1974,Widom1976,Widom1982}, and later
Fredholm determinant methods
\cite{BorodinOkounkov2000,DeiftItsZhou1997,DeiftItsKrasovsky2013}.

\medskip

\noindent\textbf{Organization.}
Section~\ref{sec:model-reduction} fixes the model, notation, and scaling
reduction to the normalized family \(I_a=[1,a]\), and proves the
Fourier--Laplace reduction to the one-dimensional trace-class operator
\(T_{a,L}\). Section~\ref{sec:trace-moments-szego} proves the trace-moment
limits and the leading Szeg\H{o} theorem. Section~\ref{sec:convexity-ldp}
proves convexity, steepness, Legendre duality, and the large deviations
theorem. Section~\ref{sec:strong-szego} proves the strong Szeg\H{o} expansion
through a canonical Wiener--Hopf comparison and a holomorphic discrepancy
limit, and derives the general bounded-interval consequences by scaling.
Section~\ref{sec:shape-family} develops analytic and asymptotic properties of
the scale-free family \(\Lambda_a\), \(J_a\), and \(\mathcal C_a\), \(a>1\).

\section{Model, scaling, and Fourier--Laplace reduction}
\label{sec:model-reduction}

This section fixes the model and notation, records the scaling reduction to the
reference intervals \(I_a=[1,a]\), \(a>1\), and derives the
Fourier--Laplace reduction to a one-dimensional trace-class operator.

\subsection{The Hardy--Szeg\H{o} zero process and projected counts}

We recall that \(Z=Z(f_{\mathbb H})\) denotes the Hardy--Szeg\H{o} zero
process on \((\mathbb H,\mathrm dA)\). It is determinantal with correlation
kernel \(K_{\mathbb H}(z,w)=-\{\pi(z-\overline w)^2\}^{-1}\),
\(z,w\in\mathbb H\), and is equivalently realized as the zero set of the
Hardy--Szeg\H{o} Gaussian analytic function introduced above. Its one-point
intensity is
\[
        K_{\mathbb H}(z,z)
        =
        \frac{1}{4\pi(\operatorname{Im}z)^2}.
\]

The covariance
\(\mathbb E[f_{\mathbb H}(z)\overline{f_{\mathbb H}(w)}]
=i/(z-\overline w)\) is invariant under simultaneous horizontal translations;
equivalently, \(K_{\mathbb H}(z+t,w+t)=K_{\mathbb H}(z,w)\) for
\(t\in\mathbb R\). Hence \(Z\) is stationary in the horizontal direction.

The zero process is also dilation-invariant: for every \(\alpha>0\),
\(\alpha Z\stackrel{\mathrm d}=Z\). Indeed,
\(\alpha^{-1/2}f_{\mathbb H}(z/\alpha)\) and \(f_{\mathbb H}(z)\) have the same
covariance kernel, hence the same law. Equivalently, this follows from the
homogeneity of \(K_{\mathbb H}\).

For a bounded interval \(I\subset(0,\infty)\), set
\(S_I:=\{x+iy\in\mathbb H:x\in\mathbb R,\ y\in I\}\). Then the projected Hardy--Szeg\H{o} process is
\[
        \mathcal X_I
        =
        \sum_{z\in Z\cap S_I}\delta_{\operatorname{Re}z}.
\]
For \(L>0\), set \(B_{I,L}:=[0,L]\times I\). By definition,
\[
N_I(L)=\mathcal X_I([0,L])=\#(Z\cap B_{I,L}).
\]
For the reference family
\(I_a=[1,a]\), \(a>1\), recall that
\(\mathcal X_a:=\mathcal X_{I_a}\) and \(N_a(L):=N_{I_a}(L)\).

Let \(I=[\alpha,\beta]\subset(0,\infty)\). From \cite{AiGuoZhou2026}, the
projected process \(\mathcal X_I\) is a simple stationary point process on
\(\mathbb R\), with intensity
\(\lambda_I=(4\pi)^{-1}(\alpha^{-1}-\beta^{-1})\). In particular,
\(\mathcal X_a\) has intensity
\[
        \lambda_a
        =
        \frac{1}{4\pi}\left(1-\frac1a\right).
\]

By the dilation invariance of the planar zero process, we obtain the following
scaling relation for projected counts.

\begin{lemma}
\label{lem:count-scaling}
Let \(I=[\alpha,\beta]\subset(0,\infty)\), and set \(a=\beta/\alpha>1\). Then, for every \(L>0\),
\[
N_I(L)\stackrel{\mathrm d}=N_a(L/\alpha),\qquad \Lambda_{I,L}(\theta)
=
\alpha^{-1}\Lambda_{a,L/\alpha}(\theta).
\]
\end{lemma}

\subsection{Laplace profiles and symbols}

For the Fourier--Laplace reduction below, we associate to a bounded interval
\(I\subset(0,\infty)\) the Laplace profile
\(H_I(s):=\int_I e^{-sy}\,\mathrm dy\) and the symbol
\(r_I(\xi):=2\xi H_I(2\xi)\), \(s,\xi>0\).

For the reference interval \(I_a=[1,a]\), we write
\[
        H_a(s):=H_{I_a}(s)
        =
        \int_1^a e^{-sy}\,\mathrm dy
        =
        \frac{e^{-s}-e^{-as}}{s},
\]
and
\begin{equation}
\label{eq:normalized-symbol}
        r_a(\xi):=r_{I_a}(\xi)
        =
        2\xi H_a(2\xi)
        =
        e^{-2\xi}-e^{-2a\xi}.
\end{equation}
If \(I=[\alpha,\beta]=\alpha[1,a]\), with \(a=\beta/\alpha\), then
\(H_I(s)=\alpha H_a(\alpha s)\) and \(r_I(\xi)=r_a(\alpha\xi)\).

We record the following basic properties of \(r_a\), used repeatedly below.

\begin{lemma}
\label{lem:basic-symbol-properties}
Let \(a>1\) and \(r_a(\xi)=e^{-2\xi}-e^{-2a\xi}\), \(\xi>0\). Then
\(0<r_a(\xi)<1\) for every \(\xi>0\), and
\[
        \rho_a:=\sup_{\xi>0}r_a(\xi)
        =
        (a-1)a^{-a/(a-1)}<1.
\]
Moreover, \(r_a\in L^1(0,\infty)\),
\begin{equation}
\label{eq:basic-symbol-integral}
        \int_0^\infty r_a(\xi)\,\mathrm d\xi
        =
        \frac12\left(1-\frac1a\right),
\end{equation}
and
\begin{equation}
\label{eq:basic-symbol-variance-integral}
        \int_0^\infty r_a(\xi)(1-r_a(\xi))\,\mathrm d\xi
        =
        \frac{(a-1)(a+3)}{4a(a+1)}.
\end{equation}
Finally, \(r_a(\xi)=2(a-1)\xi+O(\xi^2)\) as \(\xi\downarrow0\), and \(r_a\)
decays exponentially as \(\xi\to\infty\).
\end{lemma}

\begin{proof}
The inequalities \(0<r_a(\xi)<1\) are immediate from
\(0<e^{-2a\xi}<e^{-2\xi}<1\). A direct differentiation shows that \(r_a\) has
a unique maximizer and gives
\(\sup_{\xi>0}r_a(\xi)=(a-1)a^{-a/(a-1)}<1\). The integral identities
\eqref{eq:basic-symbol-integral} and \eqref{eq:basic-symbol-variance-integral}
follow by direct calculation. The expansion at the origin follows from
Taylor's formula, and exponential decay from \(0<r_a(\xi)\le e^{-2\xi}\).
\end{proof}

\subsection{The determinantal generating functional}

We now derive the Fredholm determinant representation for the projected strip count. Fix \(a>1\), \(I_a=[1,a]\), and
\(B_{a,L}:=[0,L]\times I_a\subset\mathbb H\); recall that
\(N_a(L)=\mathcal X_a([0,L])=\#(Z\cap B_{a,L})\). The reduction uses the
determinantal generating functional, the Fourier--Laplace factorization of the
Hardy--Szeg\H{o} kernel, and Sylvester's determinant identity. It transfers the planar determinant on \(L^2(B_{a,L},\mathrm dA)\) to a
determinant on \(L^2(0,\infty)\), involving the trace-class operator
\begin{equation}
\label{eq:reduced-kernel}
        T_{a,L}(\xi,\eta)
        =
        \frac{\sqrt{\xi\eta}}{\pi}
        X_L(\xi-\eta)H_a(\xi+\eta),
        \qquad \xi,\eta>0,
\end{equation}
where \(X_L(s):=\int_0^L e^{-isx}\,\mathrm dx\) and
\(H_a(s):=\int_1^a e^{-sy}\,\mathrm dy\). Throughout, inner products are conjugate-linear in the first argument and linear in the second.

We first record the generating-functional identity for strip counts in a
slightly more general joint form.

\begin{lemma}
\label{lem:generating-functional}
Let \(B_1,\ldots,B_m\subset\mathbb R\) be bounded pairwise disjoint Borel
sets, set \(A_j:=B_j\times I_a\) and \(A:=\bigcup_{j=1}^m A_j\). Let \(K_A\)
be the integral operator on \(L^2(A,\mathrm dA)\) with kernel
\(K_{\mathbb H}\) restricted to \(A\times A\). Then \(K_A\) is a positive
trace-class contraction, \(0\le K_A\le \operatorname{Id}\). Moreover, for
\(s=(s_1,\ldots,s_m)\in\mathbb C^m\), define
\[
        g_s(z):=\sum_{j=1}^m(s_j-1)\mathbf 1_{A_j}(z),
        \qquad z\in A,
\]
and let \(M_{g_s}\) be multiplication by \(g_s\) on \(L^2(A,\mathrm dA)\).
Then \(M_{g_s}K_A\) is trace class and
\begin{equation}
\label{eq:joint-pgf-determinantal}
        \mathbb E\left[\prod_{j=1}^m s_j^{\#(Z\cap A_j)}\right]
        =
        \det\left(\operatorname{Id}+M_{g_s}K_A\right),
\end{equation}
where the determinant is the Fredholm determinant on \(L^2(A,\mathrm dA)\).
\end{lemma}

\begin{proof}
Since \(A\) has finite area and is bounded away from the real axis,
\(\operatorname{tr}(K_A)=\int_A K_{\mathbb H}(z,z)\,\mathrm dA(z)<\infty\).
Thus \(K_A\) is trace class; as a compression of the correlation operator of
\(Z\), it is a positive contraction by
\cite[Theorem~4.5.5]{HoughKrishnapurPeresVirag2009}. Since \(g_s\) is bounded,
\(M_{g_s}K_A\) is trace class.

For \(s\) near \((1,\ldots,1)\), choose the principal logarithm and set
\[
        f_s(z):=-\sum_{j=1}^m(\log s_j)\mathbf 1_{A_j}(z).
\]
Then \(e^{-f_s}=1+g_s\). With \(\alpha=-1\), the determinantal generating-functional formula
\cite[Theorems 1.2 and 1.5]{ShiraiTakahashi2003} gives
\[
        \mathbb E\left[\prod_{z\in Z\cap A}(1+g_s(z))\right]
        =
        \det\left(\operatorname{Id}+M_{g_s}K_A\right).
\]
Since \(\prod_{z\in Z\cap A}(1+g_s(z))
=\prod_{j=1}^m s_j^{\#(Z\cap A_j)}\), \eqref{eq:joint-pgf-determinantal}
holds near \((1,\ldots,1)\).

Both sides are entire in \(s\). The determinant side is entire because
\(s\mapsto M_{g_s}K_A\) is affine with values in the trace-class ideal. On the
probabilistic side, \(\#(Z\cap A)\) has the law of a sum of independent
Bernoulli variables with parameters given by the eigenvalues of \(K_A\)
\cite[Theorem~4.5.3]{HoughKrishnapurPeresVirag2009}, so
\(\mathbb E R^{\#(Z\cap A)}<\infty\) for every \(R>0\), and dominated
convergence gives entire dependence on \(s\). The identity theorem completes
the proof.
\end{proof}

The single-count case of Lemma~\ref{lem:generating-functional} gives the Fredholm determinant used below.

\begin{corollary}
\label{cor:planar-count-determinant}
Let \(L>0\), and let \(K_{B_{a,L}}\) be the integral operator on \(L^2(B_{a,L},\mathrm{d}A)\) with kernel \(K_{\mathbb{H}}\) restricted to \(B_{a,L}\times B_{a,L}\). Then, for every \(\theta\in\mathbb{C}\),
\[
\mathbb{E}\exp(\theta N_a(L))
=
\det\left(\operatorname{Id}+(e^{\theta}-1)K_{B_{a,L}}\right).
\]
\end{corollary}

\begin{proof}
Apply Lemma~\ref{lem:generating-functional} with \(m=1\), \(B_1=[0,L]\), \(A_1=B_{a,L}\), and \(s_1=e^{\theta}\).
\end{proof}

\subsection{Fourier--Laplace factorization and the reduced operator}

For \(z=x+iy\in\mathbb H\) and \(\xi>0\), set
\(u_\xi(z):=\sqrt{\xi/\pi}\,e^{i\xi z}
=\sqrt{\xi/\pi}\,e^{i\xi x}e^{-\xi y}\). Then
\begin{equation}
\label{eq:fourier-laplace-kernel}
        K_{\mathbb H}(z,w)
        =
        \int_0^\infty u_\xi(z)\overline{u_\xi(w)}\,\mathrm d\xi .
\end{equation}

\begin{lemma}
\label{lem:hilbert-schmidt-factorization}
Let \(L>0\) and \(a>1\). Define
\(\mathsf A_{a,L}:L^2(B_{a,L},\mathrm dA)\to L^2(0,\infty)\) by
\[
        (\mathsf A_{a,L}f)(\xi)
        :=
        \int_{B_{a,L}}\overline{u_\xi(z)}f(z)\,\mathrm dA(z).
\]
Then \(\mathsf A_{a,L}\) is Hilbert--Schmidt,
\(K_{B_{a,L}}=\mathsf A_{a,L}^*\mathsf A_{a,L}\), and \(K_{B_{a,L}}\) is a
positive trace-class operator with \(\operatorname{tr}(K_{B_{a,L}})=L\lambda_a\).
\end{lemma}

\begin{proof}
By Tonelli's theorem,
\[
        \|\mathsf A_{a,L}\|_{\mathrm{HS}}^2
        =
        \int_0^\infty\int_0^L\int_1^a
        \frac{\xi}{\pi}e^{-2\xi y}\,\mathrm dy\,\mathrm dx\,\mathrm d\xi
        =
        L\lambda_a,
\]
so \(\mathsf A_{a,L}\) is Hilbert--Schmidt. For
\(f,g\in L^2(B_{a,L},\mathrm dA)\), Fubini's theorem and
\eqref{eq:fourier-laplace-kernel} give
\[
\begin{aligned}
\langle \mathsf A_{a,L}f,\mathsf A_{a,L}g\rangle_{L^2(0,\infty)}
       &= \int_{B_{a,L}}\int_{B_{a,L}} \overline{f(w)}K_{\mathbb{H}}(w,z)g(z) \,\mathrm{d}A(w)\,\mathrm{d}A(z)\\
        &=
        \langle K_{B_{a,L}}f,g\rangle_{L^2(B_{a,L})},    
\end{aligned}
\]
using the Hermitian symmetry of \(K_{\mathbb H}\). Hence
\(K_{B_{a,L}}=\mathsf A_{a,L}^*\mathsf A_{a,L}\). Positivity,
trace-classness, and the trace identity follow from this Hilbert--Schmidt
factorization and the norm computation above.
\end{proof}

\begin{lemma}
\label{lem:reduced-operator}
Let \(a>1\) and \(L>0\). The operator \(T_{a,L}:=\mathsf{A}_{a,L}\mathsf{A}_{a,L}^{*}\) is positive trace class on \(L^2(0,\infty)\), satisfies \(0\leq T_{a,L}\leq \operatorname{Id}\), 
\(
\operatorname{tr}(T_{a,L})=L\lambda_a,
\)
and has kernel \eqref{eq:reduced-kernel}. Moreover, for every \(\theta\in\mathbb{C}\),
\begin{equation}
\label{eq:reduced-count-determinant}
\mathbb{E}\exp(\theta N_a(L))
=
\det\left(\operatorname{Id}+(e^{\theta}-1)T_{a,L}\right).
\end{equation}
\end{lemma}

\begin{proof}
Since \(\mathsf{A}_{a,L}\) is Hilbert--Schmidt, \(T_{a,L}\) is positive trace class. Lemmas~\ref{lem:hilbert-schmidt-factorization} and~\ref{lem:generating-functional} imply \(0\leq T_{a,L}\leq\operatorname{Id}\). For \(\xi,\eta>0\),
\[
\begin{aligned}
T_{a,L}(\xi,\eta)
&=
\int_{B_{a,L}}\overline{u_{\xi}(z)}u_{\eta}(z)\,\mathrm{d}A(z) \\
&=
\frac{\sqrt{\xi\eta}}{\pi}
\left(\int_0^L e^{-i(\xi-\eta)x}\,\mathrm{d}x\right)
\left(\int_1^a e^{-(\xi+\eta)y}\,\mathrm{d}y\right),
\end{aligned}
\]
which is \eqref{eq:reduced-kernel}. The trace formula follows from the equality of the nonzero spectra of \(\mathsf{A}_{a,L}^{*}\mathsf{A}_{a,L}\) and \(\mathsf{A}_{a,L}\mathsf{A}_{a,L}^{*}\), together with Lemma~\ref{lem:hilbert-schmidt-factorization}. Finally, Corollary~\ref{cor:planar-count-determinant} and Sylvester's determinant identity give \eqref{eq:reduced-count-determinant}.
\end{proof}

\section{Trace moments and the leading Szeg\H{o} limit}
\label{sec:trace-moments-szego}

This section proves the fixed-power trace asymptotics for \(T_{a,L}\), the main
input for the leading Szeg\H{o} limit:
\[
        \lim_{L\to\infty}
        \frac{1}{L}\operatorname{tr}(T_{a,L}^m)
        =
        \frac{1}{2\pi}\int_0^\infty r_a(\xi)^m\,\mathrm d\xi,
        \qquad m\ge1.
\]
A deterministic polynomial approximation then yields the logarithmic determinant
asymptotic in Theorem~\ref{thm:szego-limit-reference}.

\subsection{A convolution estimate}

The trace-moment proof uses the following integrable envelope.

\begin{lemma}
\label{lem:convolution-bound}
Let \(\varphi(t):=(1+|t|)^{-1}\), \(t\in\mathbb R\).  For \(d\ge1\), define
\(\varphi^{*1}:=\varphi\) and \(\varphi^{*(d+1)}:=\varphi^{*d}*\varphi\).
Then the convolution integrals are finite pointwise, and for every \(d\ge1\)
there exists \(C_d<\infty\) such that
\[
        (\varphi^{*d})(s)
        \le
        C_d
        \frac{\log^{d-1}(2+|s|)}{1+|s|},
        \qquad s\in\mathbb R.
\]
\end{lemma}

\begin{proof}
By evenness it suffices to consider \(s\ge0\). The case \(d=1\) is immediate.
For \(d=2\), a standard splitting of the convolution integral over
\((-\infty,0]\), \([0,s]\), and \([s,\infty)\) gives
\((\varphi*\varphi)(s)\le C\log(2+s)/(1+s)\).

We next use the auxiliary estimate
\[
        \int_{\mathbb R}
        \frac{\log^p(2+|t|)}
        {(1+|t|)(1+|s-t|)}
        \,\mathrm dt
        \le
        C_p
        \frac{\log^{p+1}(2+s)}{1+s},
        \qquad s\ge0,\quad p\ge0.
\]
To prove it, the case \(s=0\) is immediate. For \(s>0\), split the integral
into \(|t|\le2s\) and \(|t|>2s\). On \(|t|\le2s\), we use
\(\log(2+|t|)\le C\log(2+s)\) and the \(d=2\) estimate. On \(|t|>2s\),
\(|s-t|\ge |t|/2\), and the remaining tail is bounded by
\[
        C\int_s^\infty
        \frac{\log^p(2+u)}{(1+u)^2}\,\mathrm du
        \le
        C_p\frac{\log^p(2+s)}{1+s},
\]
by integration by parts and induction in \(p\), which is stronger than needed.
This proves the auxiliary estimate.

The stated bound now follows by induction on \(d\), applying the auxiliary
estimate with \(p=d-1\). Evenness gives the estimate for all \(s\in\mathbb R\),
and the same bounds imply pointwise finiteness.
\end{proof}

\subsection{Trace-moment asymptotics}

We now prove the fixed-power trace limits.

\begin{proposition}
\label{prop:trace-moment-limit}
Let \(a>1\). For every integer \(m\ge1\),
\[
        \lim_{L\to\infty}
        \frac{1}{L}\operatorname{tr}(T_{a,L}^m)
        =
        \frac{1}{2\pi}\int_0^\infty r_a(\xi)^m\,\mathrm d\xi .
\]
\end{proposition}

\begin{proof}
The case \(m=1\) follows from Lemma~\ref{lem:reduced-operator} and
\eqref{eq:basic-symbol-integral}. Assume \(m\ge2\). Set
\(Y_L(s):=\int_{-L/2}^{L/2}e^{-isx}\,\mathrm dx\) and
\(g(u):=\int_{-1/2}^{1/2}e^{-iut}\,\mathrm dt=2\sin(u/2)/u\), with the usual
value at \(u=0\). Since \(X_L(s)=e^{-iLs/2}Y_L(s)\), the phases cancel in
cyclic products. Also \(|X_L|\le L\) and
\(H_a(\xi+\eta)\le (a-1)e^{-(\xi+\eta)}\), so the cyclic kernel product is
absolutely integrable. Hence the cyclic trace formula, applied to
\eqref{eq:reduced-kernel}, gives
\begin{align*}
\operatorname{tr}(T_{a,L}^m)
&= \frac{1}{\pi^m}
\int_{(0,\infty)^m}
\left(\prod_{j=1}^m\xi_j\right)
\left(\prod_{j=1}^mY_L(\xi_j-\xi_{j+1})\right) \\
&\quad \times
\left(\prod_{j=1}^mH_a(\xi_j+\xi_{j+1})\right)
\,\mathrm d\xi_1\cdots\mathrm d\xi_m,
\end{align*}
where \(\xi_{m+1}:=\xi_1\).

Put \(\xi:=\xi_1\), \(\delta_j:=\xi_j-\xi_{j+1}\), and
\(u_j:=L\delta_j\), \(1\le j\le m-1\). Then
\(\xi_k(\xi,u/L)=\xi-L^{-1}\sum_{j=1}^{k-1}u_j\), the Jacobian after the
\(u\)-change is \(L^{-(m-1)}\), and \(Y_L(u/L)=Lg(u)\). Thus
\begin{align*}
\frac{1}{L}\operatorname{tr}(T_{a,L}^m)
&= \frac{1}{\pi^m}
\int_0^\infty
\int_{\mathbb R^{m-1}}
\mathbf 1_{\{\xi_k(\xi,u/L)>0\ \forall k\}}
\Pi_L(\xi,u) \\
&\quad \times
\left(\prod_{j=1}^{m-1}g(u_j)\right)
g\left(-\sum_{j=1}^{m-1}u_j\right)
\,\mathrm du\,\mathrm d\xi,
\end{align*}
where
\[
\Pi_L(\xi,u)
:=
\prod_{k=1}^m
\xi_k(\xi,u/L)
H_a\left(\xi_k(\xi,u/L)+\xi_{k+1}(\xi,u/L)\right).
\]
For fixed \(\xi>0\) and \(u\in\mathbb R^{m-1}\), the indicator tends to \(1\)
and \(\Pi_L(\xi,u)\to \xi^mH_a(2\xi)^m\).

It remains to justify domination. Since \(H_a(t)\le e^{-t}/t\) for \(t>0\),
on the support of the indicator we have
\(0\le \xi_kH_a(\xi_k+\xi_{k+1})\le e^{-(\xi_k+\xi_{k+1})}\), hence
\(0\le\Pi_L(\xi,u)\le e^{-2\xi}\). Moreover,
\(|g(t)|\le C(1+|t|)^{-1}=C\varphi(t)\), and the resulting \(u\)-majorant
\[
        C_m
        \left(\prod_{j=1}^{m-1}\varphi(u_j)\right)
        \varphi\left(\sum_{j=1}^{m-1}u_j\right)
\]
is integrable on \(\mathbb R^{m-1}\) by Lemma~\ref{lem:convolution-bound}.
Dominated convergence therefore gives
\[
        \lim_{L\to\infty}
        \frac{1}{L}\operatorname{tr}(T_{a,L}^m)
        =
        \frac{1}{\pi^m}
        \left(\int_0^\infty \xi^mH_a(2\xi)^m\,\mathrm d\xi\right)U_m,
\]
where
\[
        U_m:=
        \int_{\mathbb R^{m-1}}
        \left(\prod_{j=1}^{m-1}g(u_j)\right)
        g\left(-\sum_{j=1}^{m-1}u_j\right)\,\mathrm du .
\]

It remains to compute \(U_m\). For \(\varepsilon>0\), insert the factor
\(\exp(-\varepsilon\sum_{j=1}^{m-1}u_j^2)\) in the definition of \(U_m\) and
call the resulting integral \(I_m(\varepsilon)\). The same domination gives
\(I_m(\varepsilon)\to U_m\). With
\(p_\varepsilon(s):=(4\pi\varepsilon)^{-1/2}e^{-s^2/(4\varepsilon)}\), Fubini's
theorem gives
\[
        I_m(\varepsilon)
        =
        (2\pi)^{m-1}
        \int_{-1/2}^{1/2}
        \left(
        \int_{-1/2}^{1/2}p_\varepsilon(t-t_m)\,\mathrm dt
        \right)^{m-1}
        \,\mathrm dt_m .
\]
Since \(p_\varepsilon\) is an approximate identity, the inner integral tends to
\(1\) for \(t_m\in(-1/2,1/2)\) and is bounded by \(1\). Hence
\(U_m=(2\pi)^{m-1}\). Substitution into the preceding limit, together with
\(r_a(\xi)=2\xi H_a(2\xi)\), completes the proof.
\end{proof}

\subsection{From trace moments to logarithmic determinants}

The next deterministic proposition converts fixed-power trace convergence into the leading logarithmic determinant asymptotic.

\begin{proposition}
\label{prop:logdet-szego-limit}
Let \((T_L)_{L\geq1}\) be positive trace-class contractions on a Hilbert space. Assume that
\begin{enumerate}[label=(\roman*)]
\item \(\sup_{L\geq1}L^{-1}\operatorname{tr}(T_L)<\infty\);
\item there is a measurable \(r:(0,\infty)\to[0,1]\) with \(r\in L^1(0,\infty)\) such that, for every \(m\geq1\),
\[
\lim_{L\to\infty}
\frac{1}{L}\operatorname{tr}(T_L^m)
=
\frac{1}{2\pi}
\int_0^\infty r(\xi)^m\,\mathrm{d}\xi.
\]
\end{enumerate}
Then, for each \(t>-1\),
\[
\lim_{L\to\infty}
\frac{1}{L}\log\det(\operatorname{Id}+tT_L)
=
\frac{1}{2\pi}
\int_0^\infty \log\bigl(1+t\,r(\xi)\bigr)\,\mathrm{d}\xi,
\]
locally uniformly for \(t\) in compact subsets of \((-1,\infty)\).
\end{proposition}

\begin{proof}
Fix \(t>-1\) and set \(f_t(\lambda):=\log(1+t\lambda)\) on \([0,1]\). Choose
polynomials \(q_n\to f_t'\) uniformly on \([0,1]\), and define
\(p_n(\lambda):=\int_0^\lambda q_n(s)\,\mathrm ds\). Then \(p_n(0)=0\),
\(p_n\to f_t\) in \(C^1([0,1])\), and, with \(h_n:=f_t-p_n\),
\(|h_n(\lambda)|\le \lambda\|h_n'\|_{\infty}\) for \(0\le\lambda\le1\).

Let \(\{\lambda_{L,j}\}_{j\ge1}\subset[0,1]\) be the eigenvalues of \(T_L\),
counted with multiplicity. The preceding estimate and the uniform trace bound give
\[
        \sup_{L\ge1}
        \left|
        \frac{1}{L}\operatorname{tr}(f_t(T_L))
        -
        \frac{1}{L}\operatorname{tr}(p_n(T_L))
        \right|
        \le
        \|h_n'\|_{\infty}
        \sup_{L\ge1}\frac{1}{L}\operatorname{tr}(T_L)
        \to0 .
\]
For fixed \(n\), \(p_n\) has no constant term, so the trace-moment assumptions imply
\[
        \lim_{L\to\infty}
        \frac{1}{L}\operatorname{tr}(p_n(T_L))
        =
        \frac{1}{2\pi}\int_0^\infty p_n(r(\xi))\,\mathrm d\xi .
\]
Since
\(|p_n(r(\xi))-f_t(r(\xi))|\le \|p_n'-f_t'\|_{\infty}r(\xi)\) and
\(r\in L^1(0,\infty)\), letting \(n\to\infty\) yields
\[
        \lim_{L\to\infty}
        \frac{1}{L}\operatorname{tr}(f_t(T_L))
        =
        \frac{1}{2\pi}\int_0^\infty f_t(r(\xi))\,\mathrm d\xi .
\]
For \(t>-1\), the spectral series is absolutely convergent and
\[
        \log\det(\operatorname{Id}+tT_L)
        =
        \sum_{j\ge1}\log(1+t\lambda_{L,j})
        =
        \operatorname{tr}(f_t(T_L)).
\]

For local uniformity, let \(K\subset(-1,\infty)\) be compact and set
\(m_K:=\min\{1,1+\inf_K t\}>0\). Differentiating the spectral series gives,
uniformly for \(t\in K\),
\[
        \left|
        \frac{\mathrm d}{\mathrm dt}
        \left(
        \frac{1}{L}\log\det(\operatorname{Id}+tT_L)
        \right)
        \right|
        \le
        \frac{1}{m_K}\frac{1}{L}\operatorname{tr}(T_L)
        \le C_K .
\]
Thus the normalized log-determinants are equicontinuous on \(K\). The limiting
function is continuous on \(K\), since
\[
        \left|\log(1+t r(\xi))-\log(1+s r(\xi))\right|
        \le
        \frac{|t-s|}{m_K}r(\xi),
        \qquad s,t\in K,
\]
and \(r\in L^1(0,\infty)\). Pointwise convergence, equicontinuity, and
compactness give uniform convergence on \(K\).
\end{proof}

\subsection{The leading Szeg\H{o} limit and holomorphic extension}

\begin{proposition}
\label{prop:leading-szego-limit}
Let \(a>1\). For every \(\theta\in\mathbb{R}\),
\[
\lim_{L\to\infty}
\frac{1}{L}
\log\mathbb{E}\exp(\theta N_a(L))
=
\Lambda_a(\theta),
\]
where
\[
\Lambda_a(\theta)
=
\frac{1}{2\pi}
\int_0^{\infty}
\log\left(1+(e^{\theta}-1)r_a(\xi)\right)
\,\mathrm{d}\xi.
\]
The convergence is locally uniform for \(\theta\in\mathbb{R}\).
\end{proposition}

\begin{proof}
By Lemma~\ref{lem:reduced-operator}, the exponential moment is the Fredholm determinant of \(T_{a,L}\). For real \(\theta\), set \(t=e^\theta-1>-1\). The operators \(T_{a,L}\) are positive trace-class contractions, their normalized traces are uniformly bounded, and Proposition~\ref{prop:trace-moment-limit} gives the fixed-power trace limits with the symbol \(r_a\). Proposition~\ref{prop:logdet-szego-limit} therefore yields the stated limit after substituting \(t=e^\theta-1\). Its local uniformity in \(t\) transfers to local uniformity in \(\theta\), since \(\theta\mapsto e^\theta-1\) maps compact subsets of \(\mathbb{R}\) into compact subsets of \((-1,\infty)\).
\end{proof}

\begin{proposition}
\label{prop:lambda-holomorphic-extension}
Let \(a>1\), and set \(S_{\pi}:=\{\theta\in\mathbb{C}:|\operatorname{Im}\theta|<\pi\}\). The function
\[
\Lambda_a(\theta)
=
\frac{1}{2\pi}
\int_0^{\infty}
\log\left(1+(e^{\theta}-1)r_a(\xi)\right)
\,\mathrm{d}\xi
\]
extends holomorphically to \(S_{\pi}\), where the logarithm is the branch obtained by analytic continuation from \(\theta=0\). Moreover,
\[
\Lambda_a'(0)
=
\frac{1}{4\pi}\left(1-\frac{1}{a}\right),
\qquad
\Lambda_a''(0)
=
\frac{(a-1)(a+3)}{8\pi a(a+1)}.
\]
\end{proposition}

\begin{proof}
Let \(\rho_a:=\sup_{\xi>0}r_a(\xi)\). By
Lemma~\ref{lem:basic-symbol-properties}, \(0<\rho_a<1\). For each compact
\(K\subset S_\pi\), the image of \(K\times[0,\rho_a]\) under
\((\theta,u)\mapsto 1+(e^\theta-1)u\) is disjoint from \((-\infty,0]\). Indeed,
if this image contained a non-positive real number, then either \(u=0\), which
gives \(1\), or \(u>0\), in which case the imaginary part forces
\(\sin(\operatorname{Im}\theta)=0\). Since \(\theta\in S_\pi\), this implies
\(\operatorname{Im}\theta=0\), and then \(1+(e^\theta-1)u=(1-u)+ue^\theta>0\).

Thus the logarithm is well defined on a neighborhood of this compact image.
Since it vanishes at \(1\) and is Lipschitz there,
\[
        \left|\log\left(1+(e^\theta-1)r_a(\xi)\right)\right|
        \le C_K r_a(\xi),\qquad \theta\in K,\ \xi>0.
\]
As \(r_a\in L^1(0,\infty)\), the integral converges locally uniformly on
\(S_\pi\), and Morera's theorem gives holomorphy. Differentiating under the
integral sign near \(0\), followed by
\eqref{eq:basic-symbol-integral} and
\eqref{eq:basic-symbol-variance-integral}, gives the displayed values of
\(\Lambda_a'(0)\) and \(\Lambda_a''(0)\).
\end{proof}

\begin{proof}[Proof of Theorem~\ref{thm:szego-limit-reference}]
The real-variable limit and local uniform convergence follow from Proposition~\ref{prop:leading-szego-limit}; the holomorphic extension to \(S_{\pi}\) and the formulas for \(\Lambda_a'(0)\) and \(\Lambda_a''(0)\) follow from Proposition~\ref{prop:lambda-holomorphic-extension}.
\end{proof}

\section{Convexity, Legendre duality, and large deviations}
\label{sec:convexity-ldp}

This section derives the large deviations principle from the Szeg\H{o} limit.
We analyze the convexity and slope behavior of \(\Lambda_a\), identify the
effective domain of its Legendre--Fenchel transform, and apply the
G\"artner--Ellis theorem.

\subsection{\texorpdfstring{Convexity and steepness of \(\Lambda_a\)}{Convexity and steepness of Lambda a}}

\begin{lemma}
\label{lem:convexity-steepness}
Let \(a>1\). The limiting log-moment generating function \(\Lambda_a\) belongs to \(C^\infty(\mathbb{R})\), with
\[
\Lambda_a'(\theta)
=
\frac{1}{2\pi}
\int_0^{\infty}
\frac{e^{\theta}r_a(\xi)}
{1+(e^{\theta}-1)r_a(\xi)}
\,\mathrm{d}\xi,
\]
\[
\Lambda_a''(\theta)
=
\frac{1}{2\pi}
\int_0^{\infty}
\frac{e^{\theta}r_a(\xi)(1-r_a(\xi))}
{\left(1+(e^{\theta}-1)r_a(\xi)\right)^2}
\,\mathrm{d}\xi .
\]
Moreover, \(\Lambda_a\) is strictly increasing and strictly convex on \(\mathbb{R}\), and
\begin{equation}
\label{eq:lambda-slope-limits}
\lim_{\theta\to-\infty}\Lambda_a'(\theta)=0,
\qquad
\lim_{\theta\to+\infty}\Lambda_a'(\theta)=+\infty.
\end{equation}
Consequently, the Legendre--Fenchel transform 
\(
\widetilde{J}_a(x):=\sup_{\theta\in\mathbb{R}}\{\theta x-\Lambda_a(\theta)\}
\)
is finite precisely on \([0,\infty)\), and \(J_a:=\widetilde{J}_a|_{[0,\infty)}\) is a good rate function on \([0,\infty)\).
\end{lemma}

\begin{proof}
By Lemma~\ref{lem:basic-symbol-properties}, \(0<r_a<1\),
\(r_a\in L^1(0,\infty)\), and
\(\rho_a:=\sup_{\xi>0}r_a(\xi)<1\). Fix a compact \(K\subset\mathbb R\). For
\(\theta\in K\) and \(0\le u\le\rho_a\), the denominator
\(1+(e^\theta-1)u\) is bounded away from \(0\). If
\(F_n(\theta,u):=\partial_\theta^n\log(1+(e^\theta-1)u)\), then
\(F_n(\theta,0)=0\), and the mean-value theorem gives
\(|F_n(\theta,u)|\le C_{K,n}u\). Substituting \(u=r_a(\xi)\) justifies
differentiation under the integral sign to all orders. The displayed formulas
for \(\Lambda_a'\) and \(\Lambda_a''\) follow, and their integrands are strictly
positive.

For the slope limits, when \(\theta\le0\),
\[
        0\le
        \frac{e^\theta r_a(\xi)}
        {1+(e^\theta-1)r_a(\xi)}
        \le
        \frac{r_a(\xi)}{1-\rho_a},
\]
so dominated convergence gives the left limit in
\eqref{eq:lambda-slope-limits}. As \(\theta\to+\infty\), the same integrand
increases pointwise to \(1\); monotone convergence on \((0,R)\), followed by
\(R\to\infty\), gives the right limit.

Since \(|\log(1-u)|\le C_{\rho_a}u\) for \(0\le u\le\rho_a\), the limit
\(\Lambda_a(-\infty):=\lim_{\theta\to-\infty}\Lambda_a(\theta)\) is finite.
If \(x<0\), then \(\theta x-\Lambda_a(\theta)\to+\infty\) as
\(\theta\to-\infty\). If \(x=0\), then
\(\widetilde J_a(0)=-\Lambda_a(-\infty)<\infty\). If \(x>0\), the slope limits
and strict convexity give a unique maximizer \(\theta_a(x)\) with
\(\Lambda_a'(\theta_a(x))=x\).

Finally, \(\widetilde J_a\) is lower semicontinuous as a Legendre--Fenchel
transform, and \(J_a\ge0\) on \([0,\infty)\) because \(\Lambda_a(0)=0\). For
any \(\theta_0>0\), \(J_a(x)\ge \theta_0x-\Lambda_a(\theta_0)\), so the
sublevel sets are closed and bounded in \([0,\infty)\).
\end{proof}

\subsection{Legendre duality}

We use the following standard convex-duality fact.

\begin{lemma}
\label{lem:legendre-duality}
Let \(\Lambda:\mathbb R\to\mathbb R\) be finite, \(C^1\), and strictly convex,
and define
\[
        J(x):=\sup_{\theta\in\mathbb R}\{\theta x-\Lambda(\theta)\},
        \qquad x\in\mathbb R.
\]
Set \(m_-:=\lim_{\theta\to-\infty}\Lambda'(\theta)\) and
\(m_+:=\lim_{\theta\to+\infty}\Lambda'(\theta)\), allowing infinite limits.
Then:
\begin{enumerate}[label=(\roman*)]
\item For every \(x\in(m_-,m_+)\), the supremum defining \(J(x)\) is attained
at the unique \(\theta(x)\in\mathbb R\) satisfying
\(\Lambda'(\theta(x))=x\), and
\[
J(x)=\theta(x)x-\Lambda(\theta(x)).
\]

\item The function \(J\) is differentiable on \((m_-,m_+)\), with
\(J'(x)=\theta(x)\). In particular, \(J\) is strictly convex on
\((m_-,m_+)\).

\item For every \(\theta\in\mathbb R\),
\(J(\Lambda'(\theta))=\theta\Lambda'(\theta)-\Lambda(\theta)\).

\item If \(\Lambda(0)=0\) and \(\Lambda'(0)\in(m_-,m_+)\), then
\(J(x)\ge0\) for every \(x\in\mathbb R\), and \(J\) vanishes uniquely at
\(x=\Lambda'(0)\).
\end{enumerate}
\end{lemma}

\begin{proof}
For fixed \(x\), set \(\Phi_x(\theta):=\theta x-\Lambda(\theta)\). Since
\(\Lambda\) is \(C^1\) and strictly convex, \(\Lambda'\) is continuous and
strictly increasing, the limits \(m_\pm\) exist in the extended sense, and
\(\Phi_x'(\theta)=x-\Lambda'(\theta)\).

If \(x\in(m_-,m_+)\), there is a unique \(\theta(x)\) with
\(\Lambda'(\theta(x))=x\). Then \(\Phi_x'\) is positive to the left of
\(\theta(x)\) and negative to the right, so \(\theta(x)\) is the unique
maximizer. The identity in (iii) follows by taking \(x=\Lambda'(\theta)\).

For differentiability, the maximizing property gives, for
\(x,y\in(m_-,m_+)\) with \(y>x\),
\[
        \theta(x)
        \le
        \frac{J(y)-J(x)}{y-x}
        \le
        \theta(y).
\]
Applying this with \(z<x<y\) and letting \(z\uparrow x\), \(y\downarrow x\),
using the continuity of \(\theta=(\Lambda')^{-1}\), yields \(J'(x)=\theta(x)\).
Since \(\theta\) is strictly increasing, \(J\) is strictly convex on
\((m_-,m_+)\).

Finally, if \(\Lambda(0)=0\), then \(J(x)\ge0\) by taking \(\theta=0\) in the
supremum, while (iii) with \(\theta=0\) gives \(J(\Lambda'(0))=0\). If
\(x\ne\Lambda'(0)\), then \(\theta x-\Lambda(\theta)\) has value \(0\) and
nonzero derivative at \(\theta=0\), hence is positive for \(\theta\) of a
suitable sign and sufficiently small magnitude.
\end{proof}

\begin{corollary}
\label{cor:rate-function-parametrization}
Let \(a>1\). For every \(x>0\), there exists a unique
\(\theta_a(x)\in\mathbb R\) such that
\[
        \Lambda_a'(\theta_a(x))=x,
        \qquad
        J_a(x)=\theta_a(x)x-\Lambda_a(\theta_a(x)).
\]
Moreover, \(J_a\) is differentiable and strictly convex on \((0,\infty)\), with
\(J_a'(x)=\theta_a(x)\) and
\(J_a''(x)=1/\Lambda_a''(\theta_a(x))\). Finally, \(J_a\ge0\) on
\([0,\infty)\), and it vanishes uniquely at
\(\lambda_a=\frac{1}{4\pi}(1-\frac{1}{a})\).
\end{corollary}

\begin{proof}
By Lemma~\ref{lem:convexity-steepness}, Lemma~\ref{lem:legendre-duality}
applies to \(\Lambda_a\) with \(m_-=0\) and \(m_+=+\infty\). Since
\(J_a=\widetilde J_a|_{[0,\infty)}\), this gives the parametrization,
differentiability, strict convexity, nonnegativity, and uniqueness of the zero.
As \(\Lambda_a''>0\), the inverse function theorem gives
\(\theta_a\in C^1(0,\infty)\), and differentiating
\(\Lambda_a'(\theta_a(x))=x\) yields
\(\theta_a'(x)=1/\Lambda_a''(\theta_a(x))\). Finally,
\(\Lambda_a'(0)=\lambda_a\) by \eqref{eq:lambda-a-derivatives-at-zero}.
\end{proof}

\subsection{Proof of the reference large deviations principle}

\begin{proof}[Proof of Theorem~\ref{thm:ldp-reference}]
Set \(Y_L:=N_a(L)/L\). Its normalized logarithmic moment generating function is
\(\Lambda_{a,L}\), and Theorem~\ref{thm:szego-limit-reference} gives
\(\Lambda_{a,L}(\theta)\to\Lambda_a(\theta)\) for every \(\theta\in\mathbb R\).
Since \(\Lambda_a\) is finite and differentiable on all of \(\mathbb R\), the
boundary-steepness condition in the G\"artner--Ellis theorem is vacuous.

By the G\"artner--Ellis theorem \cite[Theorem~2.3.6]{DemboZeitouni2010},
\((Y_L)_{L>0}\), viewed as \(\mathbb R\)-valued random variables, satisfies an
LDP with speed \(L\) and rate function
\[
        \widetilde J_a(x)
        :=
        \sup_{\theta\in\mathbb R}\{\theta x-\Lambda_a(\theta)\},
        \qquad x\in\mathbb R .
\]
Lemma~\ref{lem:convexity-steepness} shows that \(\widetilde J_a\) is good and
\(\widetilde J_a(x)=+\infty\) for \(x<0\). Since \(Y_L\in[0,\infty)\) almost
surely, the LDP restricts to \([0,\infty)\), with the relative topology, speed
\(L\), and rate function \(J_a=\widetilde J_a|_{[0,\infty)}\). The remaining
properties of \(J_a\) follow from Corollary~\ref{cor:rate-function-parametrization}.
\end{proof}

\begin{remark}
\label{rem:local-LDP-route}
The Brillinger-mixing estimates from \cite{AiGuoZhou2026} suggest a local
cumulant route: with \(t=e^\theta-1\), the factorial-cumulant expansion
identifies the scaled log-mgf only for small \(|t|\), hence only near
\(\theta=0\). This controls slopes only around the mean and gives at most a
local G\"artner--Ellis conclusion. The Fredholm reduction and trace asymptotics
identify \(\Lambda_a(\theta)\) for all real \(\theta\), yielding the full LDP
for \(N_a(L)/L\) on \([0,\infty)\).
\end{remark}

\section{Strong Szeg\H{o} asymptotics}
\label{sec:strong-szego}

This section proves the strong Szeg\H{o} expansion in Theorem~\ref{thm:strong-szego-reference}. The leading term came from fixed-power trace limits for \(T_{a,L}\); the \(O(1)\) correction requires a finer comparison with a canonical Wiener--Hopf model.

The key observation is that the kernel
\[
T_{a,L}(\xi,\eta)
=
\frac{\sqrt{\xi\eta}}{\pi}X_L(\xi-\eta)H_a(\xi+\eta)
\]
is close to, but not exactly, a Wiener--Hopf kernel. The obstruction is the factor \(H_a(\xi+\eta)\). We replace it by the separable diagonal approximation \(\sqrt{H_a(2\xi)H_a(2\eta)}\). These two factors have second-order contact along \(\xi=\eta\), which makes the discrepancy uniformly trace class.

Throughout this section \(a>1\) is fixed. We use the same \(H_a\), \(r_a\), and \(X_L\) as above, and regard \(r_a\) as a function on \(\mathbb{R}\) by setting \(r_a(\xi)=0\) for \(\xi\leq0\). We write \(\mathcal S_1\) for the trace-class ideal.

\subsection{The canonical Wiener--Hopf model}

We use the Fourier-transform convention
\[
\widehat{f}(\xi):=\int_{\mathbb{R}}e^{-ix\xi}f(x)\,\mathrm{d}x,
\qquad
\check{g}(x):=\frac{1}{2\pi}\int_{\mathbb{R}}e^{ix\xi}g(\xi)\,\mathrm{d}\xi.
\]

\begin{definition}
\label{def:wiener-hopf-operator}
For \(\sigma\in L^\infty(\mathbb R)\), define the truncated Wiener--Hopf operator \(W_L(\sigma)\) on \(L^2(0,L)\) by
\[
(W_L(\sigma)f)(x')
=
\frac{1}{2\pi}
\int_0^L
\int_{\mathbb R}\sigma(\xi)e^{i\xi(x'-x)}f(x)
\,\mathrm d\xi\,\mathrm dx,
\qquad 0<x'<L,
\]
in the Fourier-multiplier sense. Equivalently, if \(f\) is extended by zero to \(\mathbb R\), then \(W_L(\sigma)f\) is the restriction to \((0,L)\) of the inverse Fourier transform of \(\sigma\widehat f\).
\end{definition}

\begin{definition}
\label{def:canonical-comparison}
The canonical comparison operator \(\widetilde W_{a,L}\) is the integral operator on \(L^2(0,\infty)\) with kernel
\[
\widetilde{W}_{a,L}(\xi,\eta)
:=
\frac{\sqrt{\xi\eta}}{\pi}
X_L(\xi-\eta)
\sqrt{H_a(2\xi)H_a(2\eta)},
\qquad \xi,\eta>0.
\]
\end{definition}

\begin{lemma}
\label{lem:wh-equivalence}
The operators \(W_L(r_a)\) and \(\widetilde{W}_{a,L}\) admit a common
Hilbert--Schmidt factorization: there exists
\(V_{a,L}:L^2(0,L)\to L^2(0,\infty)\) such that
\[
        W_L(r_a)=V_{a,L}^{*}V_{a,L},
        \qquad
        \widetilde{W}_{a,L}=V_{a,L}V_{a,L}^{*}.
\]
Consequently, they are positive trace-class operators with the same nonzero
eigenvalues, counted with multiplicity, and for every \(t\in\mathbb C\),
\[
        \det(\operatorname{Id}+t\widetilde{W}_{a,L})
        =
        \det(\operatorname{Id}+tW_L(r_a)).
\]
Moreover, with \(\rho_a:=\sup_{\xi>0}r_a(\xi)\in(0,1)\),
\[
        0\leq W_L(r_a)\leq \rho_a\operatorname{Id},
        \qquad
        0\leq \widetilde{W}_{a,L}\leq \rho_a\operatorname{Id}.
\]
\end{lemma}

\begin{proof}
Define
\[
        (V_{a,L}f)(\xi)
        :=
        \sqrt{\frac{r_a(\xi)}{2\pi}}
        \int_0^L e^{-i\xi x}f(x)\,\mathrm dx,
        \qquad \xi>0.
\]
Since
\(\|V_{a,L}\|_{\mathrm{HS}}^2=(L/2\pi)\int_0^\infty r_a(\xi)\,\mathrm d\xi<\infty\),
the operator \(V_{a,L}\) is Hilbert--Schmidt. Direct computation, using the
zero extension of \(r_a\) and \eqref{eq:normalized-symbol}, gives
\(V_{a,L}^{*}V_{a,L}=W_L(r_a)\) and
\[
        (V_{a,L}V_{a,L}^{*})(\xi,\eta)
        =
        \frac{\sqrt{\xi\eta}}{\pi}
        X_L(\xi-\eta)\sqrt{H_a(2\xi)H_a(2\eta)},
\]
hence \(V_{a,L}V_{a,L}^{*}=\widetilde W_{a,L}\). The spectral equivalence of
\(V_{a,L}^{*}V_{a,L}\) and \(V_{a,L}V_{a,L}^{*}\) gives the trace-class and
eigenvalue statements, and Sylvester's identity gives the determinant identity.

For \(f\in L^2(0,L)\), extended by zero to \(\mathbb R\), Plancherel's theorem
gives
\[
        0
        \le
        \langle W_L(r_a)f,f\rangle
        =
        \frac{1}{2\pi}\int_0^\infty r_a(\xi)|\widehat f(\xi)|^2\,\mathrm d\xi
        \le
        \rho_a\|f\|_{L^2(0,L)}^2.
\]
Thus \(0\le W_L(r_a)\le \rho_a\operatorname{Id}\). Since
\(\|\widetilde W_{a,L}\|=\|V_{a,L}V_{a,L}^{*}\|
=\|V_{a,L}^{*}V_{a,L}\|=\|W_L(r_a)\|\), the same bound holds for
\(\widetilde W_{a,L}\). Finally, \(\rho_a\in(0,1)\) follows from
Lemma~\ref{lem:basic-symbol-properties}.
\end{proof}

\subsection{A trace-class criterion}

We next record a trace-class criterion tailored to kernels that vanish at
infinity and have enough \(L^2\)-control after differentiating in one variable.

\begin{lemma}
\label{lem:trace-class-criterion}
Let \(k:(0,\infty)^2\to\mathbb C\) be measurable. Assume that, for almost every
\(\eta>0\), the map \(\xi\mapsto k(\xi,\eta)\) is locally absolutely continuous
on \((0,\infty)\), satisfies \(\lim_{\xi\to\infty}k(\xi,\eta)=0\), and
\[
        A:=
        \int_0^\infty
        \sqrt u
        \left(
        \int_0^\infty |\partial_1 k(u,\eta)|^2\,\mathrm d\eta
        \right)^{1/2}
        \,\mathrm du
        <\infty .
\]
Then the integral operator \(K\) on \(L^2(0,\infty)\) with kernel \(k\) is trace
class, and \(\|K\|_1\le A\).
\end{lemma}

\begin{proof}
For each \(\delta>0\), \(A<\infty\) implies
\[
        \int_\delta^\infty
        \left(
        \int_0^\infty |\partial_1 k(u,\eta)|^2\,\mathrm d\eta
        \right)^{1/2}
        \,\mathrm du<\infty .
\]
By Minkowski's integral inequality, after a countable intersection over
rational \(\delta>0\), we may assume that
\[
\int_\delta^\infty |\partial_1 k(u,\eta)|\,\mathrm du<\infty
\]
for every such
\(\delta\) and almost every \(\eta>0\). Hence local absolute continuity and the
decay of \(k(\cdot,\eta)\) at infinity give, for almost every \((\xi,\eta)\),
\[
        k(\xi,\eta)
        =
        \int_0^\infty
        \mathbf 1_{(0,u)}(\xi)\bigl(-\partial_1 k(u,\eta)\bigr)
        \,\mathrm du .
\]

Choose a measurable representative of \(\partial_1 k\). For \(u>0\), set
\(h_u:=\mathbf 1_{(0,u)}\) and
\(g_u(\eta):=-\overline{\partial_1 k(u,\eta)}\), with \(g_u:=0\) when this
function is not in \(L^2(0,\infty)\), and define
\((K_u f)(\xi):=\langle g_u,f\rangle h_u(\xi)\). Then \(K_u\) has kernel
\(\mathbf 1_{(0,u)}(\xi)(-\partial_1 k(u,\eta))\), and
\[
        \|K_u\|_1
        =
        \|h_u\|_2\|g_u\|_2
        =
        \sqrt u
        \left(
        \int_0^\infty |\partial_1 k(u,\eta)|^2\,\mathrm d\eta
        \right)^{1/2}.
\]
Thus \(\int_0^\infty\|K_u\|_1\,\mathrm du=A\). The map \(u\mapsto K_u\) is
strongly measurable, so \(\widetilde K:=\int_0^\infty K_u\,\mathrm du\) exists
as a Bochner integral in trace norm and satisfies \(\|\widetilde K\|_1\le A\).
Since the same integral also converges in Hilbert--Schmidt norm, its kernel is
identified with \(k\) by the representation above.
\end{proof}

\subsection{Uniform trace-class discrepancy}

We now show that the difference between the reduced operator and the canonical
comparison operator is uniformly trace class.

\begin{proposition}
\label{prop:trace-class-discrepancy}
\[
        \sup_{L\ge1}\|T_{a,L}-\widetilde W_{a,L}\|_1<\infty .
\]
\end{proposition}

\begin{proof}
Let \(R_{a,L}:=T_{a,L}-\widetilde W_{a,L}\). Then
\[
        R_{a,L}(\xi,\eta)
        =
        \frac{\sqrt{\xi\eta}}{\pi}
        X_L(\xi-\eta)\Delta_a(\xi,\eta),
        \quad
        \Delta_a(\xi,\eta)
        :=
        H_a(\xi+\eta)-\sqrt{H_a(2\xi)H_a(2\eta)}.
\]
Factor \(H_a(s)=e^{-s}p_a(s)\), where
\(p_a(s):=\int_0^{a-1}e^{-sy}\,\mathrm dy\), and set
\(q_a(t):=\sqrt{p_a(2t)}\). Then
\[
        \Delta_a(\xi,\eta)
        =
        e^{-(\xi+\eta)}B_a(\xi,\eta),
        \qquad
        B_a(\xi,\eta):=p_a(\xi+\eta)-q_a(\xi)q_a(\eta).
\]
The derivatives of \(p_a\) and \(q_a\) up to order \(3\) are bounded on
\([0,\infty)\). Since \(B_a(\xi,\xi)=0\) and
\(\partial_\eta B_a(\xi,\eta)|_{\eta=\xi}=0\), Taylor's formula gives
\(B_a(\xi,\eta)=(\xi-\eta)^2s_a(\xi,\eta)\), where \(s_a\) and
\(\partial_1s_a\) are bounded. Thus, with
\(\widetilde s_a(\xi,\eta):=e^{-(\xi+\eta)}s_a(\xi,\eta)\),
\[
        \Delta_a(\xi,\eta)=(\xi-\eta)^2\widetilde s_a(\xi,\eta),
        \qquad
        |\widetilde s_a(\xi,\eta)|+|\partial_1\widetilde s_a(\xi,\eta)|
        \le C_a e^{-(\xi+\eta)} .
\]
Using \(X_L(s)(s^2)=i s(e^{-isL}-1)\), we obtain
\[
        R_{a,L}(\xi,\eta)
        =
        i\frac{\sqrt{\xi\eta}}{\pi}
        (\xi-\eta)(e^{-i(\xi-\eta)L}-1)\widetilde s_a(\xi,\eta).
\]
Define
\(
        Q_a(\xi,\eta)
        :=
        i\frac{\sqrt{\xi\eta}}{\pi}
        (\xi-\eta)\widetilde s_a(\xi,\eta),
\)
and let \(M_L\) be multiplication by \(e^{-i\xi L}\). Then
\[
R_{a,L}=M_LQ_aM_L^*-Q_a, \qquad \|R_{a,L}\|_1\le2\|Q_a\|_1.
\]
It remains to
show that \(Q_a\in\mathcal S_1\).

For fixed \(\eta>0\), the map \(\xi\mapsto Q_a(\xi,\eta)\) is locally
absolutely continuous and tends to \(0\) as \(\xi\to\infty\). The preceding
bounds imply
\[
        |\partial_1Q_a(\xi,\eta)|
        \le
        C_a(\xi^{-1/2}+1+\xi+\xi^{3/2})(1+\eta^{3/2})e^{-(\xi+\eta)}.
\]
Hence
\[
        \left(
        \int_0^\infty|\partial_1Q_a(u,\zeta)|^2\,\mathrm d\zeta
        \right)^{1/2}
        \le
        C_a(u^{-1/2}+1+u+u^{3/2})e^{-u},
\]
and the quantity in Lemma~\ref{lem:trace-class-criterion} is finite. Therefore
\(Q_a\in\mathcal S_1\), and the unitary conjugation estimate above gives the
uniform bound.
\end{proof}

\subsection{Sobolev control for the logarithmic symbol}

For \(\theta\in S_\pi\), define
\[
        \sigma_\theta(\xi):=1+(e^\theta-1)r_a(\xi),
        \qquad
        \psi_\theta(\xi):=\log\sigma_\theta(\xi),
        \qquad \xi\in\mathbb R,
\]
where \(r_a\) is extended by zero to \(\mathbb R\), and the principal logarithm
is used. The argument below shows that the logarithm is well defined; since
\(r_a(0)=0\), \(\psi_\theta\) is continuous at \(\xi=0\).

For \(0<\theta_*<\pi\), write
\(S_{\theta_*}:=\{\theta\in\mathbb C:|\operatorname{Im}\theta|<\theta_*\}\),
and equip \(W^{1,2}(\mathbb R)\) with the norm
\[
        \|f\|_{W^{1,2}(\mathbb R)}
        :=
        \left(\|f\|_{L^2(\mathbb R)}^2+\|f'\|_{L^2(\mathbb R)}^2\right)^{1/2},
\]
where \(f'\) denotes the weak derivative.

\begin{lemma}
\label{lem:log-symbol-sobolev}
Fix \(0<\theta_*<\pi\), and let \(K\subset S_{\theta_*}\) be compact. Then:
\begin{enumerate}[label=(\roman*)]
\item There exists \(\delta_K>0\) such that
\[
        \operatorname{dist}\bigl(\sigma_\theta(\xi),(-\infty,0]\bigr)
        \ge \delta_K,
        \qquad \theta\in K,\quad \xi\in\mathbb R.
\]

\item The family \(\{\psi_\theta:\theta\in K\}\) is bounded in
\(W^{1,2}(\mathbb R)\).

\item For every compact \(\Omega\subset\mathbb C\), the families
\[
        \{e^{\lambda\psi_\theta}-1:\theta\in K,\lambda\in\Omega\},
        \qquad
        \{(e^{\lambda\psi_\theta}-1)\psi_\theta:\theta\in K,\lambda\in\Omega\}
\]
are bounded in \(W^{1,2}(\mathbb R)\).

\item If
\(
        \tau_\theta(x):=(2\pi)^{-1}
        \int_{\mathbb R}\psi_\theta(\xi)e^{ix\xi}\,\mathrm d\xi,
\)
then
\[
        \mathcal C^{\operatorname{WH}}_{a,\theta_*}(\theta)
        :=
        \int_0^\infty x\tau_\theta(x)\tau_\theta(-x)\,\mathrm dx
\]
defines a holomorphic function on \(S_{\theta_*}\).
\end{enumerate}
\end{lemma}

\begin{proof}
Let \(\rho_a:=\sup_{\xi>0}r_a(\xi)\in(0,1)\). Every value of
\(\sigma_\theta(\xi)\) has the form \((1-u)+ue^\theta\), with
\(0\le u\le\rho_a\). For \(\theta\in K\), the compact image of
\(K\times[0,\rho_a]\) under this map is disjoint from \((-\infty,0]\): if
\((1-u)+ue^\theta\) were non-positive real, then \(u>0\) and the imaginary
part would force \(\sin(\operatorname{Im}\theta)=0\); since
\(K\subset S_{\theta_*}\) with \(\theta_*<\pi\), this gives
\(\operatorname{Im}\theta=0\), and then \((1-u)+ue^\theta>0\). This proves the
distance estimate.

The principal logarithm is therefore uniformly well defined on a neighborhood
of this compact image. Hence \(\sup_{\theta\in K}\|\psi_\theta\|_{L^\infty}<\infty\), and
\[
        \psi_\theta'(\xi)
        =
        \frac{(e^\theta-1)r_a'(\xi)}
        {1+(e^\theta-1)r_a(\xi)},
        \qquad \xi>0.
\]
The denominator is uniformly bounded away from zero, and \(r_a'\in L^2(0,\infty)\).
Moreover, since the logarithm vanishes at \(1\) and is Lipschitz near the
compact image, \(|\psi_\theta(\xi)|\le C_K r_a(\xi)\). As
\(r_a\in L^1(0,\infty)\cap L^2(0,\infty)\), the \(L^2\)-norms are uniformly
bounded. Finally, \(r_a(\xi)=O(\xi)\) at \(0\), so \(\psi_\theta(\xi)=O(\xi)\)
there and no boundary jump occurs.

For the third assertion, set
\(\widetilde\sigma_{\theta,\lambda}:=e^{\lambda\psi_\theta}-1\). Since
\(\lambda\in\Omega\) and \(\psi_\theta\) is uniformly bounded in \(L^\infty\),
\[
        |\widetilde\sigma_{\theta,\lambda}|\le C_{K,\Omega}|\psi_\theta|,
        \qquad
        \partial_\xi\widetilde\sigma_{\theta,\lambda}
        =
        \lambda e^{\lambda\psi_\theta}\psi_\theta'.
\]
Thus \(\widetilde\sigma_{\theta,\lambda}\) is uniformly bounded in
\(W^{1,2}(\mathbb R)\), again with no boundary contribution at \(0\). Similarly,
\[
        \partial_\xi\bigl((e^{\lambda\psi_\theta}-1)\psi_\theta\bigr)
        =
        \lambda e^{\lambda\psi_\theta}\psi_\theta'\psi_\theta
        +
        (e^{\lambda\psi_\theta}-1)\psi_\theta',
\]
and the right-hand side is uniformly bounded in \(L^2(\mathbb R)\), while
\(|(e^{\lambda\psi_\theta}-1)\psi_\theta|\le C_{K,\Omega}|\psi_\theta|\). Since
\((e^{\lambda\psi_\theta(\xi)}-1)\psi_\theta(\xi)=O(\xi^2)\) at \(0\), there is
again no boundary contribution.

For the fourth assertion, \(\psi_\theta\in L^1(\mathbb R)\), so
\(\tau_\theta\) is well defined. By Plancherel,
\[
        \int_{\mathbb R}(1+x^2)|\tau_\theta(x)|^2\,\mathrm dx
        \le C\|\psi_\theta\|_{W^{1,2}(\mathbb R)}^2
\]
locally uniformly in \(\theta\). Hence the integral defining
\(\mathcal C^{\operatorname{WH}}_{a,\theta_*}\) is locally uniformly absolutely
convergent by Cauchy--Schwarz.

It remains to note holomorphy. On compact subsets of \(S_{\theta_*}\),
\[
        \partial_\theta\psi_\theta(\xi)
        =
        \frac{e^\theta r_a(\xi)}
        {1+(e^\theta-1)r_a(\xi)},
        \qquad
        \partial_\xi\partial_\theta\psi_\theta(\xi)
        =
        \frac{e^\theta r_a'(\xi)}
        {\left(1+(e^\theta-1)r_a(\xi)\right)^2}
\]
are dominated in \(L^2\) by constant multiples of \(r_a\) and \(|r_a'|\).
Dominated convergence for difference quotients shows that
\(\theta\mapsto\psi_\theta\) is holomorphic as a \(W^{1,2}(\mathbb R)\)-valued
map. The inverse Fourier transform is continuous from \(W^{1,2}(\mathbb R)\)
to \(L^2(\mathbb R,|x|\,\mathrm dx)\), and the bilinear form
\(B(f,g):=\int_0^\infty x f(x)g(-x)\,\mathrm dx\) is continuous on this space.
Since \(\mathcal C^{\operatorname{WH}}_{a,\theta_*}(\theta)=B(\tau_\theta,\tau_\theta)\),
this completes the proof.
\end{proof}

\subsection{Product-defect estimates for Wiener--Hopf operators}

\begin{lemma}
\label{lem:wiener-hopf-product-defect}
Let \(\sigma,\tau\in L^\infty(\mathbb R)\). Assume that
\(\check\sigma,\check\tau\in L^1(\mathbb R)\) and
\[
\int_0^\infty x\left(
|\check\sigma(x)|^2+|\check\sigma(-x)|^2
+|\check\tau(x)|^2+|\check\tau(-x)|^2
\right)\,\mathrm dx<\infty .
\]
Then, for every \(L>0\), \(W_L(\sigma\tau)-W_L(\sigma)W_L(\tau)\) is trace
class on \(L^2(0,L)\), and
\[
\begin{aligned}
\|W_L(\sigma\tau)-W_L(\sigma)W_L(\tau)\|_1
&\le
\left(\int_0^\infty x|\check\sigma(x)|^2\,\mathrm dx\right)^{1/2}
\left(\int_0^\infty x|\check\tau(-x)|^2\,\mathrm dx\right)^{1/2} \\
&\quad+
\left(\int_0^\infty x|\check\sigma(-x)|^2\,\mathrm dx\right)^{1/2}
\left(\int_0^\infty x|\check\tau(x)|^2\,\mathrm dx\right)^{1/2}.
\end{aligned}
\]
\end{lemma}

\begin{proof}
Let \(J_L:L^2(0,L)\to L^2(\mathbb R)\) be extension by zero, and let
\(\Gamma_L:=J_LJ_L^*\), multiplication by \(\mathbf 1_{(0,L)}\). If
\(C_\kappa\) denotes convolution by \(\kappa\), then
\(W_L(\sigma)=J_L^*C_{\check\sigma}J_L\). Since
\(\check{\sigma\tau}=\check\sigma*\check\tau\),
\[
        W_L(\sigma\tau)-W_L(\sigma)W_L(\tau)
        =
        J_L^*C_{\check\sigma}(\operatorname{Id}-\Gamma_L)C_{\check\tau}J_L.
\]
We decompose
\(\operatorname{Id}-\Gamma_L=M_{\mathbf 1_{(-\infty,0)}}+
M_{\mathbf 1_{(L,\infty)}}\).

The first term factors through \(L^2(0,\infty)\) with Hilbert--Schmidt kernels
\(A_L^+(x,s)=\check\sigma(x+s)\) and
\(B_L^+(s,z)=\check\tau(-s-z)\), \(x,z\in(0,L)\), \(s>0\). Hence
\[
        \|A_L^+\|_{\mathrm{HS}}^2
        =
        \int_0^\infty \min\{L,u\}|\check\sigma(u)|^2\,\mathrm du
        \le
        \int_0^\infty u|\check\sigma(u)|^2\,\mathrm du,
\]
and similarly
\(\|B_L^+\|_{\mathrm{HS}}^2
\le \int_0^\infty u|\check\tau(-u)|^2\,\mathrm du\).

The second term factors through \(L^2(0,\infty)\) with kernels
\(A_L^-(x,s)=\check\sigma(x-L-s)\) and
\(B_L^-(s,z)=\check\tau(L+s-z)\), and satisfies
\[
        \|A_L^-\|_{\mathrm{HS}}^2
        \le
        \int_0^\infty u|\check\sigma(-u)|^2\,\mathrm du,
        \qquad
        \|B_L^-\|_{\mathrm{HS}}^2
        \le
        \int_0^\infty u|\check\tau(u)|^2\,\mathrm du.
\]
The claim follows from the triangle inequality and
\(\|AB\|_1\le \|A\|_{\mathrm{HS}}\|B\|_{\mathrm{HS}}\).
\end{proof}

\subsection{Uniform bounds for the normalized canonical determinant}

Define
\[
D_{a,L}(\theta)
:=
\det\left(\operatorname{Id}+(e^{\theta}-1)\widetilde{W}_{a,L}\right)
\exp\left(-L\Lambda_a(\theta)\right).
\]

\begin{lemma}
\label{lem:canonical-normalized-bounded}
Fix \(0<\theta_*<\pi\), and let \(K\subset S_{\theta_*}\) be compact. Then each \(D_{a,L}\) is holomorphic on \(S_{\theta_*}\), and
\begin{equation}
\label{eq:canonical-normalized-bound}
\sup_{L\geq 1}\sup_{\theta\in K}|D_{a,L}(\theta)|<\infty.
\end{equation}
\end{lemma}

\begin{proof}
By Lemma~\ref{lem:wh-equivalence},
\[
        \det\left(\operatorname{Id}+(e^\theta-1)\widetilde W_{a,L}\right)
        =
        \det W_L(\sigma_\theta),
        \qquad \sigma_\theta=e^{\psi_\theta},
\]
where \(\det W_L(\sigma_\theta)\) means
\(\det(\operatorname{Id}+W_L(\sigma_\theta-1))\). Since
\(\psi_\theta\in L^1(\mathbb R)\) locally uniformly in \(\theta\),
\[
\operatorname{tr}W_L(\psi_\theta)=\frac{L}{2\pi}\int_{\mathbb R}\psi_\theta(\xi)\,\mathrm d\xi
=L\Lambda_a(\theta).
\]
Hence
\[
        D_{a,L}(\theta)
        =
        \det W_L(\sigma_\theta)\exp\bigl(-\operatorname{tr}W_L(\psi_\theta)\bigr).
\]

Let \(\Omega\subset\mathbb C\) be a closed disk containing \([0,1]\). For
\(\lambda\in\Omega\), set
\[
        \sigma_{\theta,\lambda}:=e^{\lambda\psi_\theta},\qquad
        \widetilde\sigma_{\theta,\lambda}:=e^{\lambda\psi_\theta}-1,\qquad
        G_{L,\theta}(\lambda)
        :=
        W_L(\sigma_{\theta,\lambda})\exp(-W_L(\lambda\psi_\theta)).
\]
Then \(G_{L,\theta}(0)=\operatorname{Id}\), and multiplicativity of Fredholm
determinants, together with \(\det(e^T)=e^{\operatorname{tr}T}\), gives
\[
        D_{a,L}(\theta)=\det G_{L,\theta}(1).
\]
Differentiating in \(\lambda\),
\[
        \frac{\mathrm d}{\mathrm d\lambda}G_{L,\theta}(\lambda)
        =
        \left[
        W_L(\widetilde\sigma_{\theta,\lambda}\psi_\theta)
        -
        W_L(\widetilde\sigma_{\theta,\lambda})W_L(\psi_\theta)
        \right]\exp(-W_L(\lambda\psi_\theta)).
\]

By Lemma~\ref{lem:log-symbol-sobolev}, the families
\(\widetilde\sigma_{\theta,\lambda}\), \(\psi_\theta\), and
\(\widetilde\sigma_{\theta,\lambda}\psi_\theta\) are uniformly bounded in
\(W^{1,2}(\mathbb R)\) for \(\theta\in K\), \(\lambda\in\Omega\). Plancherel
and Cauchy--Schwarz give uniform \(L^1\)-bounds and the weighted \(L^2\)-bounds
required in Lemma~\ref{lem:wiener-hopf-product-defect}. Therefore
\[
        \sup_{L\ge1}\sup_{\theta\in K}\sup_{\lambda\in\Omega}
        \left\|
        W_L(\widetilde\sigma_{\theta,\lambda}\psi_\theta)
        -
        W_L(\widetilde\sigma_{\theta,\lambda})W_L(\psi_\theta)
        \right\|_1
        <\infty .
\]
Moreover,
\(\|\exp(-W_L(\lambda\psi_\theta))\|\le
\exp(\|\lambda\psi_\theta\|_{L^\infty})\), uniformly in
\(\theta\in K\), \(\lambda\in\Omega\). Thus
\[
        \sup_{L\ge1}\sup_{\theta\in K}\sup_{\lambda\in\Omega}
        \left\|
        \frac{\mathrm d}{\mathrm d\lambda}G_{L,\theta}(\lambda)
        \right\|_1
        <\infty .
\]
The same \(W^{1,2}\)-continuity and product-defect estimate give trace-norm
continuity in \(\lambda\), hence
\[
G_{L,\theta}(1)-\operatorname{Id}
=\int_0^1 \frac{\mathrm d}{\mathrm d\lambda}G_{L,\theta}(\lambda)\,\mathrm d\lambda
\]
in trace norm, uniformly bounded for \(L\ge1\), \(\theta\in K\). The estimate
\(|\det(\operatorname{Id}+A)|\le e^{\|A\|_1}\) gives
\eqref{eq:canonical-normalized-bound}.

Finally, \(\theta\mapsto(e^\theta-1)\widetilde W_{a,L}\) is trace-class
holomorphic, and \(\Lambda_a\) is holomorphic on \(S_{\theta_*}\) by the
compact-set domination
\(|\log(1+(e^\theta-1)r_a(\xi))|\le C_K r_a(\xi)\) and
\(r_a\in L^1(0,\infty)\). Thus \(D_{a,L}\) is holomorphic on \(S_{\theta_*}\).
\end{proof}

\subsection{Canonical strong Wiener--Hopf asymptotics}

We use the following form of Widom's strong Wiener--Hopf theorem. Recall that an operator \(A\) is of determinant class if \(A-\operatorname{Id}\in\mathcal S_1\); its determinant is then defined as the corresponding Fredholm determinant \cite[Chapter~IV]{GohbergKrein1969}.

\begin{theorem}[Widom {\cite[Theorem~3]{Widom1982}}]
\label{thm:widom-wh}
Let \(\sigma\in L^{\infty}(\mathbb{R})\), and let \(W_L(\sigma)\) be the truncated Wiener--Hopf operator on \(L^2(0,L)\). Suppose that there exists a bounded determination \(\psi\) of \(\log\sigma\) whose inverse Fourier transform \(\check{\psi}\) satisfies
\[
\int_{\mathbb{R}}|x|\,|\check{\psi}(x)|^2\,\mathrm{d}x<\infty.
\]
Then \(W_L(\sigma)\exp(-W_L(\psi))\) is of determinant class, and
\[
\det\bigl(W_L(\sigma)\exp(-W_L(\psi))\bigr)
\to
\exp\left(
\int_0^{\infty}x\check{\psi}(x)\check{\psi}(-x)\,\mathrm{d}x
\right)
\]
as \(L\to\infty\). If, in addition, \(\psi\in L^1(\mathbb{R})\), then
\[
\det W_L(\sigma)
\exp\left(
-\frac{L}{2\pi}\int_{\mathbb{R}}\psi(\xi)\,\mathrm{d}\xi
\right)
\to
\exp\left(
\int_0^{\infty}x\check{\psi}(x)\check{\psi}(-x)\,\mathrm{d}x
\right),
\]
where \(\det W_L(\sigma)\) denotes the Fredholm determinant of \(W_L(\sigma)=\operatorname{Id}+W_L(\sigma-1)\).
\end{theorem}

\begin{proposition}
\label{prop:canonical-strong-szego}
There exists a holomorphic function
\(
\mathcal{C}^{\operatorname{WH}}_a:S_{\pi}\to\mathbb{C}
\)
such that, as \(L\to\infty\),
\[
\log\det\left(\operatorname{Id}+(e^{\theta}-1)\widetilde{W}_{a,L}\right)
=
L\Lambda_a(\theta)+\mathcal{C}^{\operatorname{WH}}_a(\theta)+o(1)
\]
locally uniformly for \(\theta\in S_{\pi}\).
\end{proposition}

\begin{proof}
Fix \(0<\theta_*<\pi\). By Lemma~\ref{lem:log-symbol-sobolev},
\[
        \mathcal C^{\operatorname{WH}}_{a,\theta_*}(\theta)
        :=
        \int_0^\infty x\tau_\theta(x)\tau_\theta(-x)\,\mathrm dx
\]
is holomorphic on \(S_{\theta_*}\). For real
\(\theta\in(-\theta_*,\theta_*)\), the symbol
\(\sigma_\theta=e^{\psi_\theta}\) is positive and bounded away from zero,
\(\psi_\theta\in W^{1,2}(\mathbb R)\cap L^1(\mathbb R)\), and
\(\int_{\mathbb R}|x|\,|\tau_\theta(x)|^2\,\mathrm dx<\infty\). Since
\((2\pi)^{-1}\int_{\mathbb R}\psi_\theta=\Lambda_a(\theta)\),
Theorem~\ref{thm:widom-wh} and Lemma~\ref{lem:wh-equivalence} give
\[
        D_{a,L}(\theta)
        \to
        \exp\left(\mathcal C^{\operatorname{WH}}_{a,\theta_*}(\theta)\right),
        \qquad \theta\in(-\theta_*,\theta_*).
\]

By Lemma~\ref{lem:canonical-normalized-bounded}, the functions \(D_{a,L}\) are
holomorphic on \(S_{\theta_*}\) and locally bounded there, uniformly in \(L\).
Vitali's theorem, applied to arbitrary sequences \(L_n\to\infty\), and the
identity theorem therefore yield compact-uniform convergence on \(S_{\theta_*}\).

It remains to pass to logarithms. By Lemma~\ref{lem:wh-equivalence}, the
eigenvalues of \(\widetilde W_{a,L}\) lie in \([0,\rho_a]\), with
\(0<\rho_a<1\). Hence \(1+(e^\theta-1)\lambda\ne0\) for
\(\theta\in S_{\theta_*}\) and \(0\le\lambda\le\rho_a\), so the Fredholm
determinant is zero-free on \(S_{\theta_*}\). Let
\[
        G_{a,L}(\theta)
        :=
        \log\det\left(\operatorname{Id}+(e^\theta-1)\widetilde W_{a,L}\right)
        -
        L\Lambda_a(\theta)
\]
be the holomorphic branch normalized by \(G_{a,L}(0)=0\). Then
\(e^{G_{a,L}}=D_{a,L}\).

Let \(K\Subset S_{\theta_*}\), set
\(\Gamma_K:=\{t\theta:\theta\in K,\ 0\le t\le1\}\), and choose
\(U\) with \(\Gamma_K\Subset U\Subset S_{\theta_*}\). Since
\(D_{a,L}\to e^{\mathcal C^{\operatorname{WH}}_{a,\theta_*}}\) uniformly on
\(\overline U\) and the limit is zero-free, Cauchy's formula gives
\[ 
G_{a,L}' = \frac{D_{a,L}'}{D_{a,L}} \longrightarrow \left(\mathcal{C}^{\operatorname{WH}}_{a,\theta_*}\right)' 
\]
uniformly on \(\Gamma_K\).
As \(G_{a,L}(0)=0\) and
\(\mathcal C^{\operatorname{WH}}_{a,\theta_*}(0)=0\), integration along the
segments \(t\theta\) yields \(G_{a,L}\to
\mathcal C^{\operatorname{WH}}_{a,\theta_*}\) uniformly on \(K\).

If \(0<\theta_*<\theta_{**}<\pi\), the two local limits
\(\mathcal C^{\operatorname{WH}}_{a,\theta_*}\) and
\(\mathcal C^{\operatorname{WH}}_{a,\theta_{**}}|_{S_{\theta_*}}\) are
compact-uniform limits of the same functions \(G_{a,L}\), hence agree on
\(S_{\theta_*}\). The local functions therefore patch to a holomorphic function
\(\mathcal C^{\operatorname{WH}}_a\) on \(S_\pi\).
\end{proof}

\subsection{The holomorphic discrepancy determinant}

We now compare the original determinant with the canonical determinant. For
\(\theta\in S_\pi\), let
\[
F_{a,L}(\theta)
:=
\log\det\left(\operatorname{Id}+(e^\theta-1)T_{a,L}\right)
-
\log\det\left(\operatorname{Id}+(e^\theta-1)\widetilde W_{a,L}\right),
\]
where both logarithms are the holomorphic branches on \(S_\pi\) that vanish at
\(\theta=0\).

\begin{lemma}
\label{lem:discrepancy-normal-family}
The functions \(F_{a,L}\) are holomorphic on \(S_{\pi}\), satisfy \(F_{a,L}(0)=0\), and, for every compact \(K\subset S_{\pi}\),
\[
\sup_{L\geq1}\sup_{\theta\in K}|F_{a,L}(\theta)|<\infty.
\]
Consequently, \(\{F_{a,L}\}_{L\geq1}\) is a normal family on \(S_{\pi}\).
\end{lemma}

\begin{proof}
By Lemmas~\ref{lem:reduced-operator} and~\ref{lem:wh-equivalence},
\(T_{a,L}\) and \(\widetilde W_{a,L}\) are positive trace-class contractions.
Thus each spectral factor has the form \(1+(e^\theta-1)\lambda\), with
\(0\le\lambda\le1\), and is nonzero on \(S_\pi\). Since the eigenvalues are
summable, the Fredholm products converge absolutely and the determinants are
zero-free on \(S_\pi\). As \(S_\pi\) is simply connected and both determinants
equal \(1\) at \(\theta=0\), the normalized logarithms exist and are
holomorphic.

Fix compact \(K\subset S_\pi\), and set
\(E_{L,s}:=(1-s)\widetilde W_{a,L}+sT_{a,L}\), \(0\le s\le1\). Then
\(E_{L,s}\) is a positive trace-class contraction. Hence
\[
        m_K:=
        \inf_{\theta\in K,\ 0\le\lambda\le1}
        |1+(e^\theta-1)\lambda|
        >0
\]
and
\[
        \left\|
        \left(\operatorname{Id}+(e^\theta-1)E_{L,s}\right)^{-1}
        \right\|
        \le m_K^{-1},
        \qquad \theta\in K,\quad 0\le s\le1.
\]
Let \(H_{L,s}(\theta):=\log\det(\operatorname{Id}+(e^\theta-1)E_{L,s})\) be
the branch normalized by \(H_{L,s}(0)=0\). Since
\(s\mapsto E_{L,s}\) is affine in trace norm and the determinants stay
nonzero, the logarithmic differential formula for Fredholm determinants; see, e.g., \cite[Corollary~5.2]{Simon2005}, with the endpoint branches matched by normalization at \(\theta=0\), gives
\[
        F_{a,L}(\theta)
        =
        \int_0^1
        \operatorname{tr}
        \left[
        \left(\operatorname{Id}+(e^\theta-1)E_{L,s}\right)^{-1}
        (e^\theta-1)(T_{a,L}-\widetilde W_{a,L})
        \right]
        \,\mathrm ds .
\]
Therefore, for \(\theta\in K\),
\[
        |F_{a,L}(\theta)|
        \le
        m_K^{-1}\sup_{\zeta\in K}|e^\zeta-1|\,
        \|T_{a,L}-\widetilde W_{a,L}\|_1 .
\]
The uniform bound follows from Proposition~\ref{prop:trace-class-discrepancy},
and Montel's theorem gives normality.
\end{proof}

The normal-family argument identifies the discrepancy limit through its Taylor
coefficients at \(0\). We first record a determinant derivative formula.

\begin{lemma}
\label{lem:stirling-trace-derivatives}
Let \(A\) be a trace-class operator. For \(|\theta|\) sufficiently small, define
\[
        H_A(\theta):=\log\det\left(\operatorname{Id}+(e^\theta-1)A\right),
\]
where the logarithm is the holomorphic branch near \(0\) normalized by
\(H_A(0)=0\). Then, for every integer \(n\ge1\),
\begin{equation}
\label{eq:stirling-logdet-derivative-formula}
        H_A^{(n)}(0)
        =
        \sum_{m=1}^n
        (-1)^{m+1}(m-1)!S(n,m)\operatorname{tr}(A^m),
\end{equation}
where \(S(n,m)\) is the Stirling number of the second kind. Consequently,
\begin{equation}
\label{eq:stirling-trace-derivative-formula}
        F_{a,L}^{(n)}(0)
        =
        \sum_{m=1}^n
        (-1)^{m+1}(m-1)!S(n,m)
        \left(
        \operatorname{tr}(T_{a,L}^m)-\operatorname{tr}(\widetilde W_{a,L}^m)
        \right).
\end{equation}
\end{lemma}

\begin{proof}
For \(|z|\|A\|<1\), the Fredholm logarithm has the trace-norm expansion
\[
        \log\det(\operatorname{Id}+zA)
        =
        \sum_{m=1}^\infty
        \frac{(-1)^{m+1}}{m}\operatorname{tr}(A^m)z^m .
\]
Substituting \(z=e^\theta-1\), using local uniform convergence near
\(\theta=0\), and the identity
\[
        (e^\theta-1)^m
        =
        m!\sum_{k=m}^\infty S(k,m)\frac{\theta^k}{k!},
\]
only terms \(m\le n\) contribute to \(H_A^{(n)}(0)\), giving
\eqref{eq:stirling-logdet-derivative-formula}. Applying this to
\(A=T_{a,L}\) and \(A=\widetilde W_{a,L}\), and subtracting the two identities,
gives \eqref{eq:stirling-trace-derivative-formula}.
\end{proof}

\subsection{Trace-defect constants and the holomorphic limit}

For \(m\ge2\) and \(u=(u_1,\ldots,u_{m-1})\in\mathbb R^{m-1}\), set
\(d_1(u):=-(u_1+\cdots+u_{m-1})\) and \(d_j(u):=u_{j-1}\), \(2\le j\le m\).
Also define
\[
        \ell_m(u)
        :=
        \max_{0\le k\le m-1}(u_1+\cdots+u_k)
        -
        \min_{0\le k\le m-1}(u_1+\cdots+u_k),
\]
where the empty sum is understood as \(0\). With \(y_0:=y_m\), define
\[
        G_{T,m}(u)
        :=
        \frac{1}{\pi^m}
        \int_{[1,a]^m}
        \prod_{j=1}^m
        \frac{1}{(y_{j-1}+y_j+i d_j(u))^2}
        \,\mathrm dy_1\cdots\mathrm dy_m
\]
and
\[
        G_{W,m}(u)
        :=
        \frac{1}{\pi^m}
        \int_{[1,a]^m}
        \prod_{j=1}^m
        \frac{1}{(2z_j+i d_j(u))^2}
        \,\mathrm dz_1\cdots\mathrm dz_m.
\]
Write \(D_m(u):=G_{T,m}(u)-G_{W,m}(u)\).

\begin{proposition}
\label{prop:trace-defect-limit}
For every \(m\ge2\), one has \(\ell_mD_m\in L^1(\mathbb R^{m-1})\), and
\[
        \lim_{L\to\infty}
        \left(
        \operatorname{tr}(T_{a,L}^m)-\operatorname{tr}(\widetilde W_{a,L}^m)
        \right)
        =
        \Gamma_a^{(m)},
        \qquad
        \Gamma_a^{(m)}
        :=
        -\int_{\mathbb R^{m-1}}\ell_m(u)D_m(u)\,\mathrm du .
\]
For \(m=1\), the trace discrepancy vanishes identically, and
\(\Gamma_a^{(1)}:=0\).
\end{proposition}

\begin{proof}
The bounds \(H_a(\xi+\eta)\le(a-1)e^{-(\xi+\eta)}\) and
\(H_a(2\xi)\le(a-1)e^{-2\xi}\) make the kernels Hilbert--Schmidt and the
cyclic products absolutely integrable. Hence the cyclic trace formula applies
to both \(T_{a,L}\) and \(\widetilde W_{a,L}\). Expanding \(X_L\) and \(H_a\),
applying Fubini, and changing variables \(x=x_1\),
\(u_j=x_{j+1}-x_j\), \(1\le j\le m-1\), gives
\[
        \operatorname{tr}(T_{a,L}^m)
        =
        L\int_{\mathbb R^{m-1}}w_{m,L}(u)G_{T,m}(u)\,\mathrm du,
\]
\[
        \operatorname{tr}(\widetilde W_{a,L}^m)
        =
        L\int_{\mathbb R^{m-1}}w_{m,L}(u)G_{W,m}(u)\,\mathrm du,
\]
where
\[
        w_{m,L}(u)
        :=
        \frac{1}{L}\int_0^L
        \prod_{k=1}^{m-1}
        \mathbf 1_{[0,L]}(x+u_1+\cdots+u_k)\,\mathrm dx .
\]

Uniformly in \(y_j,z_j\in[1,a]\),
\[
        \left|
        \frac{1}{(y_{j-1}+y_j+i d_j(u))^2}
        \right|
        +
        \left|
        \frac{1}{(2z_j+i d_j(u))^2}
        \right|
        \le
        \frac{C_a}{1+d_j(u)^2}.
\]
Thus
\[
        |G_{T,m}(u)|+|G_{W,m}(u)|
        \le
        C_{a,m}
        \frac{1}{1+(u_1+\cdots+u_{m-1})^2}
        \prod_{j=1}^{m-1}\frac{1}{1+u_j^2}.
\]
This gives \(D_m\in L^1\). Since \(\ell_m(u)\le |u_1|+\cdots+|u_{m-1}|\),
the same majorant is \(\ell_m\)-integrable: fixing \(k\) and writing
\(c=\sum_{j\ne k}u_j\), the integral
\[
        \int_{\mathbb R}
        \frac{|v|}
        {(1+v^2)(1+(v+c)^2)}\,\mathrm dv
\]
is uniformly bounded in \(c\), and the remaining variables are integrable.

The admissible set of \(x\)'s has length \(\max\{0,L-\ell_m(u)\}\), so
\(w_{m,L}(u)=\max\{0,1-\ell_m(u)/L\}\). Hence
\(|L(w_{m,L}(u)-1)|\le\ell_m(u)\) and
\(L(w_{m,L}(u)-1)\to-\ell_m(u)\). Subtracting the two overlap representations,
\[
        \operatorname{tr}(T_{a,L}^m)-\operatorname{tr}(\widetilde W_{a,L}^m)
        =
        L\Delta_a^{(m)}
        +
        \int_{\mathbb R^{m-1}}L(w_{m,L}(u)-1)D_m(u)\,\mathrm du,
\]
where \(\Delta_a^{(m)}:=\int_{\mathbb R^{m-1}}D_m(u)\,\mathrm du\). Dominated
convergence gives the second term as \(\Gamma_a^{(m)}+o(1)\).

It remains to show \(\Delta_a^{(m)}=0\). For \(z\) near \(0\), set
\[
g_L(z):=L^{-1}\log\det(\operatorname{Id}+z\widetilde W_{a,L}).
\]
Applying
Proposition~\ref{prop:canonical-strong-szego} with
\(\theta=\log(1+z)\) gives locally uniform convergence
\[
        g_L(z)
        \to
        \frac{1}{2\pi}\int_0^\infty\log(1+zr_a(\xi))\,\mathrm d\xi .
\]
By Cauchy's formula and the local Fredholm logarithm expansion,
\[
        g_L^{(m)}(0)
        =
        (-1)^{m+1}(m-1)!L^{-1}\operatorname{tr}(\widetilde W_{a,L}^m).
\]
Thus \(L^{-1}\operatorname{tr}(\widetilde W_{a,L}^m)\) has the same limit as
\(L^{-1}\operatorname{tr}(T_{a,L}^m)\) in
Proposition~\ref{prop:trace-moment-limit}. Dividing the preceding expansion by
\(L\) and letting \(L\to\infty\) gives \(\Delta_a^{(m)}=0\).

For \(m=1\), Lemma~\ref{lem:reduced-operator} gives
\(\operatorname{tr}(T_{a,L})=L\lambda_a\), while
Lemma~\ref{lem:wh-equivalence} gives
\[
\operatorname{tr}(\widetilde W_{a,L})=\operatorname{tr}(W_L(r_a))
=(L/2\pi)\int_0^\infty r_a(\xi)\,\mathrm d\xi=L\lambda_a.
\]
\end{proof}

\begin{remark}
For \(m\geq2\), the cancellation \(\Delta_a^{(m)}=0\) can also be checked directly from the explicit formulas for \(G_{T,m}\) and \(G_{W,m}\). With the sign convention above, both integrals over \(\mathbb R^{m-1}\) equal
\[
\frac{m!}{4\pi}
\int_{[1,a]^m}
\frac{\mathrm{d}y_1\cdots\mathrm{d}y_m}
{(y_1+\cdots+y_m)^{m+1}}.
\]
The proof above avoids this direct computation: the canonical Wiener--Hopf asymptotic and the trace-moment limit show that the linear-in-\(L\) parts of the two traces agree, forcing \(\Delta_a^{(m)}=0\).
\end{remark}

\begin{proposition}
\label{prop:discrepancy-limit}
There exists a holomorphic function \(\Phi_a:S_{\pi}\to\mathbb{C}\) such that
\(F_{a,L}\to \Phi_a\) locally uniformly on \(S_{\pi}\). It satisfies
\(\Phi_a(0)=0\), and for every integer \(n\geq1\),
\begin{equation}
\label{eq:discrepancy-taylor-coefficients}
\Phi_a^{(n)}(0)
=
\sum_{m=1}^n
(-1)^{m+1}(m-1)!S(n,m)\Gamma_a^{(m)}.
\end{equation}
\end{proposition}

\begin{proof}
Fix \(0<\theta_*<\pi\). By Lemma~\ref{lem:discrepancy-normal-family},
\(\{F_{a,L}\}_{L\ge1}\) is normal on \(S_{\theta_*}\). Let \(L_j\to\infty\),
and pass to a subsequence such that \(F_{a,L_j}\to G\) compact-uniformly on
\(S_{\theta_*}\). Then \(G\) is holomorphic, and Cauchy's formula gives
\(F_{a,L_j}^{(n)}(0)\to G^{(n)}(0)\) for every \(n\ge0\). Since
\(F_{a,L}(0)=0\), also \(G(0)=0\). For \(n\ge1\),
Lemma~\ref{lem:stirling-trace-derivatives} and
Proposition~\ref{prop:trace-defect-limit} yield
\[
        G^{(n)}(0)
        =
        \sum_{m=1}^n
        (-1)^{m+1}(m-1)!S(n,m)\Gamma_a^{(m)}.
\]
Thus all subsequential limits have the same Taylor expansion at \(0\); by the
identity theorem they agree on \(S_{\theta_*}\). Hence the full family
converges compact-uniformly on \(S_{\theta_*}\).

Denote the local limit by \(\Phi_{a,\theta_*}\). If
\(0<\theta_*<\theta_{**}<\pi\), then
\(\Phi_{a,\theta_*}\) and
\(\Phi_{a,\theta_{**}}|_{S_{\theta_*}}\) are limits of the same functions on
\(S_{\theta_*}\), so they agree. The local limits glue to a holomorphic
function \(\Phi_a:S_\pi\to\mathbb C\), and every compact subset of \(S_\pi\) is
contained in some \(S_{\theta_*}\).
\end{proof}

\subsection{Completion of the strong Szeg\H{o} expansion}
\label{proof:strong-szego-reference}

\begin{proof}[Proof of Theorem~\ref{thm:strong-szego-reference}]
Set
\(\mathcal C_a(\theta):=\mathcal C_a^{\operatorname{WH}}(\theta)+\Phi_a(\theta)\),
\(\theta\in S_\pi\). By Propositions~\ref{prop:canonical-strong-szego}
and~\ref{prop:discrepancy-limit}, \(\mathcal C_a\) is holomorphic on \(S_\pi\).

By Lemma~\ref{lem:reduced-operator},
\[
        \mathbb E e^{\theta N_a(L)}
        =
        \det\left(\operatorname{Id}+(e^\theta-1)T_{a,L}\right).
\]
Since \(T_{a,L}\) is a positive trace-class contraction, every spectral factor
has the form \(1+(e^\theta-1)\lambda\), \(0\le\lambda\le1\), and is nonzero on
\(S_\pi\). Thus the Fredholm determinant is zero-free there, and we take the
holomorphic logarithm normalized to vanish at \(\theta=0\).

By definition of \(F_{a,L}\),
\[
        \log\det\left(\operatorname{Id}+(e^\theta-1)T_{a,L}\right)
        =
        \log\det\left(\operatorname{Id}+(e^\theta-1)\widetilde W_{a,L}\right)
        +
        F_{a,L}(\theta).
\]
The canonical expansion from Proposition~\ref{prop:canonical-strong-szego} and
the compact-uniform convergence \(F_{a,L}\to\Phi_a\) from
Proposition~\ref{prop:discrepancy-limit} complete the proof.
\end{proof}

\begin{remark}
The proof gives the structural decomposition
\(\mathcal C_a=\mathcal C_a^{\operatorname{WH}}+\Phi_a\). The first term is
the canonical Wiener--Hopf correction
\[
        \mathcal C_a^{\operatorname{WH}}(\theta)
        =
        \int_0^\infty x\tau_\theta(x)\tau_\theta(-x)\,\mathrm dx,
\]
where 
\[
\tau_\theta(x)
        =
        \frac{1}{2\pi}\int_{\mathbb R}
        \log\left(1+(e^\theta-1)r_a(\xi)\right)e^{ix\xi}\,\mathrm d\xi,
\]
and \(r_a\) is extended by zero to \((-\infty,0]\). The second term is the
locally uniform holomorphic limit
\[
        \Phi_a(\theta)
        =
        \lim_{L\to\infty}
        \left[
        \log\det\left(\operatorname{Id}+(e^\theta-1)T_{a,L}\right)
        -
        \log\det\left(\operatorname{Id}+(e^\theta-1)\widetilde W_{a,L}\right)
        \right],
\]
whose Taylor coefficients are given by
\eqref{eq:discrepancy-taylor-coefficients}. In particular,
\(\Phi_a(0)=0\) and \(\mathcal C_a(0)=0\).
\end{remark}

\subsection{Scaling back to general bounded intervals}
\label{proof:general-interval-scaling}

\begin{proof}[Proof of Corollary~\ref{cor:general-interval-scaling}]
Let \(I=[\alpha,\beta]\subset(0,\infty)\) and \(a=\beta/\alpha>1\). By
Lemma~\ref{lem:count-scaling} and Theorem~\ref{thm:szego-limit-reference},
\(\Lambda_I(\theta)=\alpha^{-1}\Lambda_a(\theta)\) for real \(\theta\). Since
both sides are holomorphic on \(S_\pi\), the identity holds throughout
\(S_\pi\). Also
\[
        J_I(x)
        =
        \sup_{\theta\in\mathbb R}\{\theta x-\Lambda_I(\theta)\}
        =
        \alpha^{-1}J_a(\alpha x).
\]

The LDP transfers under the scaling map. Indeed, with \(M=L/\alpha\),
\[
        \frac{N_I(L)}{L}
        \stackrel{\mathrm d}{=}
        \frac{1}{\alpha}\frac{N_a(M)}{M}.
\]
Applying Theorem~\ref{thm:ldp-reference} to \(\alpha A\) and using
\(L=\alpha M\) gives the LDP on \([0,\infty)\) with speed \(L\) and rate
function \(J_I\). Goodness follows from that of \(J_a\).

Finally, Theorem~\ref{thm:strong-szego-reference} and the same distributional
identity give, locally uniformly for \(\theta\in S_\pi\),
\[
        \log\mathbb E e^{\theta N_I(L)}
        =
        \log\mathbb E e^{\theta N_a(L/\alpha)}
        =
        \frac{L}{\alpha}\Lambda_a(\theta)+\mathcal C_a(\theta)+o(1)
        =
        L\Lambda_I(\theta)+\mathcal C_a(\theta)+o(1).
\]
Hence \(\mathcal C_I(\theta)=\mathcal C_a(\theta)\).
\end{proof}

\begin{remark}
The scaling corollary separates shape from scale. The ratio
\(a=\beta/\alpha\) is the shape parameter, while \(\alpha\) only rescales the
leading objects:
\[
        \Lambda_I(\theta)=\alpha^{-1}\Lambda_a(\theta),
        \qquad
        J_I(x)=\alpha^{-1}J_a(\alpha x).
\]
The strong correction is scale-invariant,
\(\mathcal C_I(\theta)=\mathcal C_a(\theta)\), which is why
\(I_a=[1,a]\) is the natural primary family.
\end{remark}

\section{Analytic and asymptotic properties of the shape family}
\label{sec:shape-family}

The main theorems identify \(\Lambda_a\), \(J_a\), and \(\mathcal C_a\) for
the normalized family \(I_a=[1,a]\). This section records further analytic and
asymptotic properties of these objects.  We first note that the explicit
formula for \(a=2\) from the introduction has a natural extension to integer
shape parameters: although general integer \(a\ge2\) no longer yields such an
elementary closed form, it still admits a dilogarithmic representation.  We
then turn to Taylor coefficients and analyticity of \(\Lambda_a\), parameter
regimes in \(a\) and their effect on \(\Lambda_a\) and \(J_a\), and low-order
structural information for \(\mathcal C_a\).

\begin{example}
\label{ex:integer-shape-explicit}
For integer \(a\ge2\), this representation is
\[
\Lambda_a(\theta)
=
-\frac{1}{4\pi}
\sum_{j=1}^{a}\operatorname{Li}_2(-\omega_j(\theta)),
\qquad \theta\in S_\pi,
\]
where \(\omega_1(\theta),\ldots,\omega_a(\theta)\) are the roots, counted with
multiplicity, of
\(
\omega^a-(e^\theta-1)\omega^{a-1}+(-1)^{a+1}(e^\theta-1)=0.
\)
The expression is understood by analytic continuation from \(\theta=0\).  For
real \(\theta\), the corresponding rate function is parametrized by
\[
J_a(x(\theta))=\theta x(\theta)-\Lambda_a(\theta),
\qquad
x(\theta)=
\frac{e^\theta}{4\pi}
\int_0^1
\frac{1-t^{a-1}}{1+(e^\theta-1)(t-t^a)}\,\mathrm dt .
\]
It follows from the substitution \(t=e^{-2\xi}\), the factorization of
\(1+(e^\theta-1)(t-t^a)\), and the standard dilogarithm primitive.
\end{example}

\subsection{\texorpdfstring{Power integrals of the symbol \(r_a\)}{Power integrals of the symbol r a}}

The powers of \(r_a\) occur in trace expansions, Taylor coefficients, and
cumulant calculations. Recall that
\[\rho_a:=\sup_{\xi>0}r_a(\xi)=(a-1)a^{-a/(a-1)}.\]

\begin{lemma}
\label{lem:moment-integrals}
Let \(a>1\), and for \(m\geq1\) set
\(
        M_m(a):=\int_0^\infty r_a(\xi)^m\,\mathrm d\xi.
\)
Then
\begin{equation}
\label{eq:moment-integrals-beta}
        M_m(a)
        =
        \frac{1}{2(a-1)}
        B\left(\frac{m}{a-1},m+1\right),
\end{equation}
where \(B\) is the Beta function. Equivalently,
\begin{equation}
\label{eq:moment-integrals-binomial}
        M_m(a)
        =
        \frac{1}{2}
        \sum_{j=0}^{m}
        (-1)^j
        \binom{m}{j}
        \frac{1}{m+(a-1)j}.
\end{equation}
Moreover, as \(m\to\infty\),
\begin{equation}
\label{eq:moment-integrals-asymptotic}
        M_m(a)
        \sim
        \sqrt{\frac{\pi}{2a}}\,m^{-1/2}\rho_a^m.
\end{equation}
\end{lemma}

\begin{proof}
With \(t=e^{-2\xi}\), 
\[
        M_m(a)
        =
        \frac12\int_0^1 t^{m-1}(1-t^{a-1})^m\,\mathrm dt.
\]
The change of variables \(s=t^{a-1}\) gives
\eqref{eq:moment-integrals-beta}, while expanding the integrand before the
change gives \eqref{eq:moment-integrals-binomial}. Finally,
\[
        M_m(a)
        =
        \frac{1}{2(a-1)}
        \frac{
        \Gamma\left(\frac{m}{a-1}\right)\Gamma(m+1)
        }{
        \Gamma\left(\frac{am}{a-1}+1\right)
        },
\]
and Stirling's formula gives \eqref{eq:moment-integrals-asymptotic}.
\end{proof}

\subsection{\texorpdfstring{Taylor coefficients and Taylor radius of \(\Lambda_a\)}{Taylor coefficients and Taylor radius of Lambda a}}

We next identify the Taylor coefficients of \(\Lambda_a\) at the origin and
determine the exact radius of convergence. The answer is governed by the
maximum value \(\rho_a\) of the symbol.

\begin{proposition}
\label{prop:taylor-coefficients}
Let \(a>1\), define
\[
\Sigma_a
:=
\left\{
\theta\in\mathbb{C}:
1+(e^{\theta}-1)r\in(-\infty,0]
\text{ for some } r\in[0,\rho_a]
\right\}.
\]
Then \(\Lambda_a\) extends holomorphically to \(\mathbb{C}\setminus\Sigma_a\). For every integer \(n\geq1\),
\[
\Lambda_a^{(n)}(0)
=
\frac{1}{4\pi(a-1)}
\sum_{m=1}^{n}
(-1)^{m+1}(m-1)!S(n,m)
B\left(\frac{m}{a-1},m+1\right),
\]
where \(S(n,m)\) is the Stirling number of the second kind. The Taylor radius \(R_a\) of \(\Lambda_a\) at \(0\) is
\[
R_a
=
\operatorname{dist}(0,\Sigma_a)
=
\sqrt{
\pi^2+
\left(
\max\left\{
\log\frac{1-\rho_a}{\rho_a},
0
\right\}
\right)^2
}.
\]
Equivalently,
\[
R_a
=
\begin{cases}
\pi,
&
a\geq a_c,\\[2mm]
\sqrt{\pi^2+\log^2\left(\frac{1-\rho_a}{\rho_a}\right)},
&
1<a<a_c,
\end{cases}
\]
where \(a_c>1\) is the unique solution of \((a_c-1)a_c^{-a_c/(a_c-1)}=1/2\).
\end{proposition}

\begin{proof}
Write \(\theta=x+iy\). For \(r\in[0,\rho_a]\),
\(1+(e^\theta-1)r=(1-r)+re^x(\cos y+i\sin y)\). This lies in
\((-\infty,0]\) only if \(y=(2k+1)\pi\) and \(e^x\ge(1-r)/r\). Hence, with
\(x_a:=\log((1-\rho_a)/\rho_a)\),
\[
        \Sigma_a
        =
        \bigcup_{k\in\mathbb Z}
        \{x+(2k+1)\pi i:x\ge x_a\}.
\]

Let \(K\Subset\mathbb C\setminus\Sigma_a\). The compact image of
\(K\times[0,\rho_a]\) under \((\theta,r)\mapsto 1+(e^\theta-1)r\) is disjoint
from \((-\infty,0]\). Thus the principal logarithm is holomorphic near this
image, and
\[
        \left|\log(1+(e^\theta-1)r_a(\xi))\right|
        \le C_K r_a(\xi),
        \qquad \theta\in K,\ \xi>0.
\]
Since \(r_a\in L^1(0,\infty)\), the integral converges locally uniformly on
\(\mathbb C\setminus\Sigma_a\), and Morera's theorem gives the holomorphic
extension.

Choose \(\eta>0\) such that
\(\rho_a\sup_{|\theta|<\eta}|e^\theta-1|<1\). The logarithmic series is then
dominated by an \(L^1\)-multiple of \(r_a\), and near \(0\),
\[
        \Lambda_a(\theta)
        =
        \frac{1}{2\pi}
        \sum_{m=1}^\infty
        \frac{(-1)^{m+1}}{m}(e^\theta-1)^m
        \int_0^\infty r_a(\xi)^m\,\mathrm d\xi .
\]
Using
\(\left.\frac{\mathrm d^n}{\mathrm d\theta^n}(e^\theta-1)^m\right|_{\theta=0}
=m!S(n,m)\), only \(m\le n\) contributes to \(\Lambda_a^{(n)}(0)\), and
Lemma~\ref{lem:moment-integrals} gives the displayed coefficient formula.

The holomorphic extension gives \(R_a\ge\operatorname{dist}(0,\Sigma_a)\). We
prove the reverse inequality. First suppose \(x_a\ge0\), and set
\(\theta_*:=x_a+i\pi\). Let \(\xi_a\) be the unique maximizer of \(r_a\). Since
\(r_a(\xi_a)=\rho_a\) and \(r_a''(\xi_a)<0\), there are
\(\varepsilon,c_1,c_2>0\) such that
\[
        \rho_a-c_2u^2
        \le
        r_a(\xi_a+u)
        \le
        \rho_a-c_1u^2,
        \qquad |u|<\varepsilon.
\]
For \(x<x_a\), differentiating under the integral sign along \(x+i\pi\) gives
\[
        \Lambda_a'(x+i\pi)
        =
        -\frac{e^x}{2\pi}
        \int_0^\infty
        \frac{r_a(\xi)}
        {1-(1+e^x)r_a(\xi)}
        \,\mathrm d\xi.
\]
With \(\delta_x:=1-(1+e^x)\rho_a\downarrow0\) as \(x\uparrow x_a\), the
quadratic maximum yields
\[
        |\Lambda_a'(x+i\pi)|
        \ge
        c\int_{-\varepsilon}^{\varepsilon}
        \frac{\mathrm du}{\delta_x+Cu^2}
        \to\infty.
\]
Thus \(\Lambda_a\) cannot extend holomorphically through \(\theta_*\), and
\(R_a\le|\theta_*|=\sqrt{x_a^2+\pi^2}\).

Now suppose \(x_a<0\). Fix \(x\in(x_a,0)\) and set
\(\lambda_x:=(1+e^x)^{-1}<\rho_a\). Then
\(E_x:=\{\xi>0:r_a(\xi)>\lambda_x\}\) is a nonempty bounded open interval. For
\(\theta_\varepsilon^\pm:=x+i(\pi\mp\varepsilon)\), the two logarithms
\(\log(1+(e^{\theta_\varepsilon^\pm}-1)r_a(\xi))\) approach the negative real
axis from opposite sides on \(E_x\), and their difference tends to \(2\pi i\)
there, while it tends to \(0\) for almost every \(\xi\notin E_x\).

Choose a compact interval \(J\subset(0,\infty)\) containing \(\overline{E_x}\)
in its interior. On \(J\), dominated convergence applies, since the two
logarithms are conjugate and their difference is uniformly bounded by \(2\pi\).
On \(J^c\), the arguments stay uniformly away from the branch cut, and the
Lipschitz bound \(C\varepsilon r_a(\xi)\) gives a vanishing contribution.
Consequently,
\[
        \lim_{\varepsilon\downarrow0}
        \left(
        \Lambda_a(\theta_\varepsilon^+)-\Lambda_a(\theta_\varepsilon^-)
        \right)
        =
        i|E_x|\ne0.
\]
Thus \(x+i\pi\) is a genuine singular point for every \(x\in(x_a,0)\). Taking
the infimum over \(x\in(x_a,0)\) gives \(R_a\le\pi\).

Combining the two cases gives \(R_a=\operatorname{dist}(0,\Sigma_a)\). The
nearest points of \(\Sigma_a\) to \(0\) lie on \(\operatorname{Im}\theta=\pm\pi\)
with real part \(\max\{x_a,0\}\), 
\[
R_a
=
\sqrt{\pi^2+\left(\max\{x_a,0\}\right)^2}
=
\sqrt{
\pi^2+
\left(
\max\left\{
\log\frac{1-\rho_a}{\rho_a},
0
\right\}
\right)^2
}.
\]
Finally,
\(\frac{\mathrm d}{\mathrm da}\log\rho_a=(\log a)/(a-1)^2>0\), while
\(\rho_a\to0\) as \(a\downarrow1\) and \(\rho_a\to1\) as \(a\to\infty\). Hence
there is a unique \(a_c>1\) with \(\rho_{a_c}=1/2\), and \(x_a\le0\) is
equivalent to \(a\ge a_c\).
\end{proof}

\begin{remark}
The set \(\Sigma_a\) is the obstruction to analytic continuation of \(\Lambda_a\) across the logarithmic branch cut. The transition at \(a=a_c\) occurs when \(\rho_a=\sup r_a\) crosses \(1/2\). Numerically, \(a_c\approx4.403497879062293\). For \(a\geq a_c\), the nearest obstruction lies at imaginary distance \(\pi\) from the origin, so the Taylor radius is exactly \(\pi\). Representative values are
\[
\renewcommand{\arraystretch}{1.25}
\begin{array}{c|c|c|c}
a & \rho_a & x_a & R_a\\
\hline
2 & 0.25 & \log 3 & \sqrt{\pi^2+\log^2 3} \\
4 & 3\cdot 4^{-4/3} & 0.110229903285\ldots & 3.143525891839\ldots \\
a_c & 0.5 & 0 & \pi \\
6 & 5\cdot 6^{-6/5} & -0.332452271263\ldots & \pi
\end{array}
\]
where \(x_a=\log((1-\rho_a)/\rho_a)\).
\end{remark}

\subsection{Dependence on the shape parameter}

We record the monotonicity of \(\Lambda_a\) in \(a\), together with the
thin-window limit \(a\downarrow1\) and the large-window limit \(a\to\infty\).

\begin{proposition}
\label{prop:monotonicity-thin-large-a}
\begin{enumerate}[label=\rm(\roman*)]

\item For each real \(\theta\), the map \(a\mapsto\Lambda_a(\theta)\) belongs to \(C^1((1,\infty))\), and
\[
\partial_a\Lambda_a(\theta)
=
\frac{e^{\theta}-1}{\pi}
\int_0^{\infty}
\frac{\xi e^{-2a\xi}}
{1+(e^{\theta}-1)(e^{-2\xi}-e^{-2a\xi})}
\,\mathrm{d}\xi.
\]
In particular, \(\partial_a\Lambda_a(\theta)\) has the sign of \(\theta\).

\item Let \(\delta:=a-1\). For every compact \(K\subset\mathbb{C}\), as \(\delta\downarrow0\),
\[
\Lambda_a(\theta)
=
\frac{\delta}{4\pi}(e^{\theta}-1)
-
\frac{\delta^2}{32\pi}
(e^{2\theta}+6e^{\theta}-7)
+
O_K(\delta^3)
\]
uniformly for \(\theta\in K\). The first-order term is of Poisson type.

\item For every compact \(K\subset S_{\pi}\), as \(a\to\infty\),
\[
\Lambda_a(\theta)
=
\Lambda_{\infty}(\theta)
-
\frac{1}{4\pi a}
\operatorname{Li}_2(1-e^{-\theta})
+
O_K(a^{-2})
\]
uniformly for \(\theta\in K\), where
\[
\Lambda_{\infty}(\theta)
:=
\frac{1}{2\pi}
\int_0^{\infty}
\log(1+(e^{\theta}-1)e^{-2\xi})
\,\mathrm{d}\xi
=
-\frac{1}{4\pi}
\operatorname{Li}_2(1-e^{\theta}).
\]
Here \(\operatorname{Li}_2\) denotes the principal branch on
\(\mathbb C\setminus[1,\infty)\).
\end{enumerate}
\end{proposition}

\begin{proof}
For real \(\theta\),
\[
        1+(e^\theta-1)(e^{-2\xi}-e^{-2a\xi})
        =
        1-r_a(\xi)+e^\theta r_a(\xi)
        \ge \min\{1,e^\theta\}.
\]
Differentiation under the integral sign is justified on compact
\(a\)-intervals by domination with \(C\xi e^{-2a_0\xi}\), giving the displayed
formula for \(\partial_a\Lambda_a\). Its sign is that of \(e^\theta-1\).

For the thin-window regime, set \(z=e^\theta-1\) and \(\delta=a-1\). Since
\(0\le r_a(\xi)\le 2\delta\xi e^{-2\xi}\), uniformly for \(\theta\in K\),
\[
        \log(1+zr_a)=zr_a-\frac{z^2}{2}r_a^2+O_K(r_a^3).
\]
Using
\(\int_0^\infty r_a=\delta/2-\delta^2/2+O(\delta^3)\),
\(\int_0^\infty r_a^2=\delta^2/8+O(\delta^3)\), and
\(\int_0^\infty r_a^3=O(\delta^3)\), we obtain
\[
        \Lambda_a(\theta)
        =
        \frac{z\delta}{4\pi}
        -
        \frac{z\delta^2}{4\pi}
        -
        \frac{z^2\delta^2}{32\pi}
        +
        O_K(\delta^3),
\]
which is the stated expansion.

For the large-\(a\) regime, again put \(z=e^\theta-1\). The change
\(t=e^{-2\xi}\) gives
\[
        \Lambda_a(\theta)
        =
        \frac{1}{4\pi}
        \int_0^1
        \frac{\log(1+z(t-t^a))}{t}\,\mathrm dt,
        \qquad
        \Lambda_\infty(\theta)
        =
        \frac{1}{4\pi}
        \int_0^1
        \frac{\log(1+zt)}{t}\,\mathrm dt.
\]
In \(\Lambda_\infty-\Lambda_a\), substituting \(s=t^a\) and then
\(s=e^{-u}\) yields
\[
        \Lambda_\infty(\theta)-\Lambda_a(\theta)
        =
        \frac{1}{4\pi a}
        \int_0^\infty G_\theta(a^{-1},u)\,\mathrm du,
\]
where
\[
        G_\theta(v,u)
        =
        \log(1+ze^{-vu})
        -
        \log(1+ze^{-vu}-ze^{-u}).
\]
For compact \(K\subset S_\pi\), all logarithms stay uniformly away from the
branch cut, and
\[
        \partial_vG_\theta(v,u)
        =
        \frac{uz^2e^{-vu}e^{-u}}
        {(1+ze^{-vu})(1+ze^{-vu}-ze^{-u})},
        \qquad
        |\partial_vG_\theta(v,u)|\le C_Kue^{-u}.
\]
Hence \(G_\theta(a^{-1},u)=G_\theta(0,u)+O_K(a^{-1}ue^{-u})\), and
\[
        \Lambda_\infty(\theta)-\Lambda_a(\theta)
        =
        \frac{1}{4\pi a}
        \int_0^\infty G_\theta(0,u)\,\mathrm du
        +
        O_K(a^{-2}).
\]
Since
\(G_\theta(0,u)=-\log(1-(1-e^{-\theta})e^{-u})\), the standard representation
\[
        \operatorname{Li}_2(w)
        =-\int_0^\infty\log(1-we^{-u})\,\mathrm du
\]
gives
\(\int_0^\infty G_\theta(0,u)\,\mathrm du=\operatorname{Li}_2(1-e^{-\theta})\).
The formula for \(\Lambda_\infty\) follows from the same representation.
\end{proof}

\begin{corollary}
\label{cor:identifiability}
Let \(a,b>1\). If \(\Lambda_a(\theta)=\Lambda_b(\theta)\) for all \(\theta\) in a nonempty open subset of \(S_{\pi}\), then \(a=b\). Moreover, if \(I=[\alpha,\beta]\) with \(0<\alpha<\beta<\infty\), then the pair
\(
\left(\Lambda_I'(0),\Lambda_I''(0)\right)
\)
determines \(I\) uniquely.
\end{corollary}

\begin{proof}
By holomorphy, equality on a nonempty open subset of \(S_\pi\) implies equality
on all of \(S_\pi\). Evaluating derivatives at \(0\) gives
\((4\pi)^{-1}(1-a^{-1})=(4\pi)^{-1}(1-b^{-1})\), hence \(a=b\).

For \(I=[\alpha,\beta]\), set \(a=\beta/\alpha\). The scaling relation
\(\Lambda_I=\alpha^{-1}\Lambda_a\) gives
\[
        \Lambda_I'(0)
        =
        \frac{1}{\alpha}\frac{a-1}{4\pi a},
        \qquad
        \Lambda_I''(0)
        =
        \frac{1}{\alpha}\frac{(a-1)(a+3)}{8\pi a(a+1)}.
\]
Thus \(\Lambda_I''(0)/\Lambda_I'(0)=(a+3)/(2(a+1))\), and this strictly
decreasing function determines \(a\). Then \(\Lambda_I'(0)\) determines
\(\alpha\), and \(\beta=\alpha a\).
\end{proof}

\subsection{Regularity and endpoint behavior of the rate function}

We now turn to the Legendre transform \(J_a\). It is analytic in the interior
of its effective domain, while at the hard endpoint \(x=0\) it has the
characteristic \(x\log x\) singularity.

\begin{proposition}
\label{prop:rate-function-regularity}
The rate function \(J_a\) is real analytic and strictly convex on
\((0,\infty)\). Moreover, there exist \(\varepsilon_a>0\) and a function
\(h_a\), holomorphic on \(\{z\in\mathbb C:|z|<\varepsilon_a\}\), such that
\[
        J_a(x)
        =
        J_a(0)+x\log x+h_a(x),
        \qquad 0<x<\varepsilon_a,
\]
and, with the principal branch of the logarithm,
\(z\mapsto J_a(0)+z\log z+h_a(z)\) is holomorphic on
\(\{0<|z|<\varepsilon_a\}\setminus(-\varepsilon_a,0]\). In particular, as
\(x\downarrow0\),
\[
        J_a(x)
        =
        J_a(0)+x\log x+\gamma_a x+O(x^2)
\]
for some real constant \(\gamma_a\).
\end{proposition}

\begin{proof}
For \(x_0>0\), let \(\theta_0\in\mathbb R\) be the unique point satisfying
\(\Lambda_a'(\theta_0)=x_0\). Since \(\Lambda_a\) is holomorphic on \(S_\pi\)
and \(\Lambda_a''(\theta_0)>0\), the holomorphic inverse function theorem gives
a local holomorphic inverse \(\theta_a(x)\) of \(\Lambda_a'\). Hence, near
\(x_0\),
\[
        J_a(x)=x\theta_a(x)-\Lambda_a(\theta_a(x)),
        \qquad
        J_a''(x)=\frac{1}{\Lambda_a''(\theta_a(x))}>0.
\]

It remains to analyze \(x\downarrow0\). Set \(q_a:=r_a/(1-r_a)\). Since
\(0\le r_a\le\rho_a<1\) and \(r_a\in L^1(0,\infty)\), we have
\(q_a\in L^1(0,\infty)\cap L^\infty(0,\infty)\). Thus
\[
        \Psi_a(z)
        :=
        \frac{1}{2\pi}\int_0^\infty \log(1+zq_a(\xi))\,\mathrm d\xi
\]
is holomorphic near \(0\). For real \(\theta\) sufficiently negative, with
\(z=e^\theta\),
\[
        \Lambda_a(\theta)
        =
        \Lambda_a(-\infty)+\Psi_a(z),
        \qquad
        \Lambda_a'(\theta)=z\Psi_a'(z),
\]
where
\(\Lambda_a(-\infty)=(2\pi)^{-1}\int_0^\infty\log(1-r_a(\xi))\,\mathrm d\xi\).

Let \(G_a(z):=z\Psi_a'(z)\). Since
\(\Psi_a'(0)=(2\pi)^{-1}\int_0^\infty q_a(\xi)\,\mathrm d\xi=:\kappa_a>0\),
we have \(G_a(z)=\kappa_a z+O(z^2)\). Hence \(G_a\) has a local holomorphic
inverse \(g_a\), and \(g_a(w)/w\) is holomorphic and nonvanishing near \(0\).
After shrinking the disk, write \(g_a(w)=w e^{L_a(w)}\), with \(L_a\)
holomorphic and real-valued on real \(w\) near \(0\). For small \(x>0\), the
equation \(x=\Lambda_a'(\theta_a(x))\) gives
\[
        e^{\theta_a(x)}=g_a(x),
        \qquad
        \theta_a(x)=\log x+L_a(x).
\]

Since \(J_a'(x)=\theta_a(x)\) on \((0,\infty)\) and
\(J_a(0)=-\Lambda_a(-\infty)\), the preceding expansion yields
\[
        J_a(x)
        =
        J_a(0)+x\log x-x+\int_0^x L_a(t)\,\mathrm dt .
\]
Thus \(h_a(z):=-z+\int_0^z L_a(w)\,\mathrm dw\) is holomorphic near \(0\). The
slit-disk continuation follows from the principal branch of \(\log z\), and
\(h_a(x)=\gamma_a x+O(x^2)\) with \(\gamma_a\in\mathbb R\).
\end{proof}

\subsection{Asymptotics of the rate function}

The asymptotic expansions for \(\Lambda_a\) transfer to \(J_a\) by Legendre
duality. The thin-window regime is Poissonian at leading order.

\begin{proposition}
\label{prop:rate-function-asymptotics}
Let \(J_{\infty}\) be the Legendre--Fenchel transform of \(\Lambda_{\infty}\),
and write \(\delta:=a-1>0\). Then:
\begin{enumerate}[label=\rm(\roman*)]

\item For every compact \(K\subset(0,\infty)\), as \(\delta\downarrow0\),
\[
        J_a(\delta x)
        =
        \delta J_{\operatorname{Pois}}(x)
        +
        \delta^2\,\frac{16\pi^2x^2+24\pi x-7}{32\pi}
        +
        O_K(\delta^3)
\]
uniformly for \(x\in K\), where \(J_{\operatorname{Pois}}\) is the Cram\'er
transform of the Poisson law with mean \(1/(4\pi)\),
\[
        J_{\operatorname{Pois}}(x)
        =
        x\log(4\pi x)-x+\frac{1}{4\pi},
        \qquad x>0.
\]

\item For every compact \(K\subset(0,\infty)\), if
\(\theta_{\infty}(x)\) is defined by
\(\Lambda_{\infty}'(\theta_{\infty}(x))=x\), then, as \(a\to\infty\),
\[
        J_a(x)
        =
        J_{\infty}(x)
        +
        \frac{1}{4\pi a}
        \operatorname{Li}_2 \bigl(1-e^{-\theta_{\infty}(x)}\bigr)
        +
        O_K(a^{-2})
\]
uniformly for \(x\in K\).
\end{enumerate}
\end{proposition}

\begin{proof}
For \rm(i), Proposition~\ref{prop:monotonicity-thin-large-a}\rm(ii),
together with Cauchy's formula on slightly larger compact sets, gives,
uniformly on compact real \(\theta\)-intervals,
\[
        \Lambda_a(\theta)
        =
        \delta\,\frac{e^\theta-1}{4\pi}
        -
        \delta^2\,\frac{e^{2\theta}+6e^\theta-7}{32\pi}
        +
        O(\delta^3),
\]
\[
        \Lambda_a'(\theta)
        =
        \delta\,\frac{e^\theta}{4\pi}
        -
        \delta^2\,\frac{e^{2\theta}+3e^\theta}{16\pi}
        +
        O(\delta^3).
\]
For \(x>0\), the leading maximizer is
\(\theta_*(x):=\log(4\pi x)\), and
\[
        J_{\operatorname{Pois}}(x)
        =
        x\theta_*(x)-\frac{e^{\theta_*(x)}-1}{4\pi}
        =
        x\log(4\pi x)-x+\frac{1}{4\pi}.
\]
On each compact \(K\subset(0,\infty)\), the implicit function theorem applied
to \(\delta^{-1}\Lambda_a'(\theta)\) gives the maximizer
\(\vartheta_a(x)=\theta_*(x)+O_K(\delta)\) of \(J_a(\delta x)\). Hence
\[
        J_a(\delta x)
        =
        \delta\left(x\vartheta_a(x)-\frac{e^{\vartheta_a(x)}-1}{4\pi}\right)
        +
        \delta^2\,\frac{e^{2\vartheta_a(x)}+6e^{\vartheta_a(x)}-7}{32\pi}
        +
        O_K(\delta^3).
\]
Since the first derivative of the leading objective vanishes at
\(\theta_*(x)\), replacing \(\vartheta_a(x)\) by \(\theta_*(x)\) gives
\[
        J_a(\delta x)
        =
        \delta J_{\operatorname{Pois}}(x)
        +
        \delta^2\,\frac{16\pi^2x^2+24\pi x-7}{32\pi}
        +
        O_K(\delta^3).
\]

For \rm(ii), Proposition~\ref{prop:monotonicity-thin-large-a}\rm(iii),
together with Cauchy's formula on slightly larger compact sets, gives the
differentiated expansion uniformly on compact real \(\theta\)-intervals.
Moreover
\[
        \Lambda_\infty'(\theta)
        =
        \frac{\theta e^\theta}{4\pi(e^\theta-1)},
        \qquad
        \Lambda_\infty''(\theta)
        =
        \frac{e^\theta(e^\theta-1-\theta)}{4\pi(e^\theta-1)^2}>0,
\]
with continuous values at \(\theta=0\). Thus
\(\Lambda_\infty':\mathbb R\to(0,\infty)\) is a strictly increasing bijection.
For compact \(K\subset(0,\infty)\), the implicit function theorem gives the
maximizer \(\theta_a(x)\) of \(J_a(x)\), with
\(\theta_a(x)=\theta_\infty(x)+O_K(a^{-1})\). Therefore
\[
        J_a(x)
        =
        x\theta_a(x)-\Lambda_\infty(\theta_a(x))
        +
        \frac{1}{4\pi a}
        \operatorname{Li}_2\bigl(1-e^{-\theta_a(x)}\bigr)
        +
        O_K(a^{-2}).
\]
Taylor expansion of \(\theta\mapsto x\theta-\Lambda_\infty(\theta)\) at
\(\theta_\infty(x)\), using
\(x=\Lambda_\infty'(\theta_\infty(x))\), gives
\[
        x\theta_a(x)-\Lambda_\infty(\theta_a(x))
        =
        J_\infty(x)+O_K(a^{-2}).
\]
Replacing \(\theta_a(x)\) by \(\theta_\infty(x)\) in the \(a^{-1}\)-coefficient
changes the expression only by \(O_K(a^{-2})\), which gives the stated
large-\(a\) expansion.
\end{proof}

\subsection{Quadratic strong correction and variance asymptotics}

We finish by computing the quadratic term of the strong Szeg\H{o} correction
\(\mathcal C_a=\mathcal C_a^{\operatorname{WH}}+\Phi_a\).

\begin{lemma}
\label{lem:wh-quadratic-term}
For every \(a>1\), \(\mathcal{C}^{\operatorname{WH}}_a(0)=0\),
\(\left(\mathcal{C}^{\operatorname{WH}}_a\right)'(0)=0\), and
\begin{equation}
\label{eq:wh-correction-second-derivative}
\left(\mathcal{C}^{\operatorname{WH}}_a\right)''(0)
=
\frac{a-1}{2\pi^2(a+1)}\log a.
\end{equation}
Equivalently, as \(\theta\to0\) in every smaller strip \(S_{\theta_*}\),
\(0<\theta_*<\pi\),
\[
\mathcal{C}^{\operatorname{WH}}_a(\theta)
=
\frac{a-1}{4\pi^2(a+1)}(\log a)\theta^2
+
O(|\theta|^3).
\]
\end{lemma}

\begin{proof}
Let \(r(\xi):=r_a(\xi)\mathbf 1_{(0,\infty)}(\xi)\). Since
\(\psi_\theta=\log(1+(e^\theta-1)r)\) is holomorphic as a
\(W^{1,2}(\mathbb R)\)-valued function, \(\psi_\theta=\theta r+O(|\theta|^2)\)
near \(0\). Define
\[
        \kappa_a(x)
        :=
        \frac{1}{2\pi}\int_0^\infty r_a(\xi)e^{ix\xi}\,\mathrm d\xi .
\]
If \(\tau_\theta\) is the inverse Fourier transform of \(\psi_\theta\), then
\(\tau_\theta=\theta\kappa_a+O(|\theta|^2)\) in
\(L^2(\mathbb R,|x|\,\mathrm dx)\). By continuity of the bilinear form
\((f,g)\mapsto\int_0^\infty x f(x)g(-x)\,\mathrm dx\),
\[
        \mathcal C_a^{\operatorname{WH}}(\theta)
        =
        \theta^2
        \int_0^\infty x\kappa_a(x)\kappa_a(-x)\,\mathrm dx
        +
        O(|\theta|^3).
\]

It remains to compute the integral. Since
\[
        \kappa_a(x)
        =
        \frac{1}{2\pi}
        \left(\frac{1}{2-ix}-\frac{1}{2a-ix}\right)
        =
        \frac{a-1}{\pi(2-ix)(2a-ix)},
\]
we have
\[
        \int_0^\infty x\kappa_a(x)\kappa_a(-x)\,\mathrm dx
        =
        \frac{(a-1)^2}{\pi^2}
        \int_0^\infty
        \frac{x\,\mathrm dx}{(x^2+4)(x^2+4a^2)}
        =
        \frac{a-1}{4\pi^2(a+1)}\log a.
\]
The zeroth and first derivatives vanish by the same expansion.
\end{proof}

\begin{lemma}
\label{lem:discrepancy-quadratic-term}
For every \(a>1\), \(\Phi_a(0)=0\), \(\Phi_a'(0)=0\), and
\begin{equation}
\label{eq:discrepancy-correction-second-derivative}
\Phi_a''(0)
=
\frac{1}{\pi^2}\log\frac{(a+1)^2}{4a}
-
\frac{a-1}{2\pi^2(a+1)}\log a.
\end{equation}
\end{lemma}

\begin{proof}
Proposition~\ref{prop:discrepancy-limit}, applied with \(n=1,2\), gives
\[
        \Phi_a'(0)=\Gamma_a^{(1)}=0,
        \qquad
        \Phi_a''(0)=\Gamma_a^{(1)}-\Gamma_a^{(2)}=-\Gamma_a^{(2)}.
\]
For \(m=2\), using the evenness of \(G_{T,2}\) and \(G_{W,2}\), 
\[
        \int_0^\infty uG_{T,2}(u)\,\mathrm du
        =
        \frac{1}{2\pi^2}\log\frac{(a+1)^2}{4a},
        \qquad
        \int_0^\infty uG_{W,2}(u)\,\mathrm du
        =
        \frac{a-1}{4\pi^2(a+1)}\log a.
\]
Hence
\[
        \Gamma_a^{(2)}
        =
        \frac{a-1}{2\pi^2(a+1)}\log a
        -
        \frac{1}{\pi^2}\log\frac{(a+1)^2}{4a},
\]
and the formula for \(\Phi_a''(0)\) follows.
\end{proof}

\begin{proposition}
\label{prop:correction-second-order}
For every \(a>1\), \(\mathcal C_a(0)=0\), \(\mathcal C_a'(0)=0\), and
\(
\mathcal C_a''(0)
=
\frac{1}{\pi^2}\log\frac{(a+1)^2}{4a}.
\)
Equivalently, as \(\theta\to0\) in every smaller strip \(S_{\theta_*}\),
\(0<\theta_*<\pi\),
\[
\mathcal C_a(\theta)
        =
        \frac{1}{2\pi^2}
        \log\frac{(a+1)^2}{4a}\,\theta^2
        +
        O(|\theta|^3).
\]
\end{proposition}

\begin{proof}
By the decomposition
\(\mathcal C_a=\mathcal C_a^{\operatorname{WH}}+\Phi_a\),
Lemmas~\ref{lem:wh-quadratic-term} and~\ref{lem:discrepancy-quadratic-term}
give \(\mathcal C_a(0)=0\) and \(\mathcal C_a'(0)=0\). For the second
derivative,
\[
        \mathcal C_a''(0)
        =
        \left(\mathcal C_a^{\operatorname{WH}}\right)''(0)
        +
        \Phi_a''(0).
\]
Substituting \eqref{eq:wh-correction-second-derivative} and
\eqref{eq:discrepancy-correction-second-derivative}, the terms proportional to
\((a-1)(a+1)^{-1}\log a\) cancel. Since \(\mathcal C_a\) is holomorphic on
\(S_\pi\), Taylor's theorem gives the quadratic expansion.
\end{proof}

\begin{remark}
In Proposition~\ref{prop:correction-second-order}, the terms proportional to
\((a-1)(a+1)^{-1}\log a\) cancel between the Wiener--Hopf correction and the
discrepancy correction. Thus the quadratic coefficient of the full correction
reduces to the simpler shape-dependent quantity
\(\pi^{-2}\log((a+1)^2/(4a))\).
\end{remark}

\begin{corollary}
\label{cor:variance-asymptotics}
For every \(a>1\),
\[
        \operatorname{Var}(N_a(L))
        =
        L\frac{(a-1)(a+3)}{8\pi a(a+1)}
        +
        \frac{1}{\pi^2}\log\frac{(a+1)^2}{4a}
        +
        o(1).
\]
More generally, for \(I=[\alpha,\beta]\), \(0<\alpha<\beta\), and
\(a=\beta/\alpha\),
\[
        \operatorname{Var}(N_I(L))
        =
        L\alpha^{-1}
        \frac{(a-1)(a+3)}{8\pi a(a+1)}
        +
        \frac{1}{\pi^2}\log\frac{(a+1)^2}{4a}
        +
        o(1).
\]
\end{corollary}

\begin{proof}
Let \(B_{a,L}=[0,L]\times[1,a]\), then
\[
        \operatorname{Var}(N_a(L))
        =
        \int_{B_{a,L}}K_{\mathbb H}(z,z)\,\mathrm dA(z)
        -
        \int_{B_{a,L}}\int_{B_{a,L}}
        |K_{\mathbb H}(z,w)|^2\,\mathrm dA(z)\mathrm dA(w).
\]
The first term is \(L\lambda_a=L(a-1)/(4\pi a)\). For the second term, changing variables
\(u=x-x'\) gives
\[
        \int_{B_{a,L}}\int_{B_{a,L}}
        |K_{\mathbb H}(z,w)|^2\,\mathrm dA(z)\mathrm dA(w)
        =
        \int_{-L}^L (L-|u|)G_{T,2}(u)\,\mathrm du,
\]
where
\[
        G_{T,2}(u)
        =
        \frac{1}{\pi^2}
        \int_{[1,a]^2}
        \frac{\mathrm dy_1\,\mathrm dy_2}
        {\bigl((y_1+y_2)^2+u^2\bigr)^2}.
\]
Since \(|u|G_{T,2}(u)\in L^1(\mathbb R)\), dominated convergence yields
\[
\begin{aligned}
\int_{-L}^L (L-|u|)G_{T,2}(u)\,\mathrm du
        &=
        L\int_{\mathbb R}G_{T,2}(u)\,\mathrm du
        -
        \int_{\mathbb R}|u|G_{T,2}(u)\,\mathrm du
        +
        o(1)\\
        &=L\frac{(a-1)^2}{8\pi a(a+1)}-\frac{1}{\pi^2}\log\frac{(a+1)^2}{4a}+o(1).
\end{aligned}
\]
Combining these identities with
\(\lambda_a=(a-1)/(4\pi a)\) gives the reference-interval formula.

For \(I=[\alpha,\beta]\), Lemma~\ref{lem:count-scaling} gives
\(N_I(L)\stackrel{\mathrm d}=N_a(L/\alpha)\), where \(a=\beta/\alpha\).
Substituting \(L/\alpha\) in the reference-interval asymptotic gives the
general formula.
\end{proof}

\end{document}